\documentclass[a4paper, 11pt]{amsart}   
\usepackage{amssymb,amscd,latexsym}   
\usepackage{amsmath}
\usepackage{amsthm}
\usepackage{mathdots}
\usepackage[pagebackref,colorlinks=true,linkcolor=blue,urlcolor=blue]{hyperref}
\usepackage{color}
\usepackage{tabularx}
\usepackage{amsfonts}
\usepackage{paralist}
\usepackage{aliascnt}
\usepackage[initials]{amsrefs}
\usepackage{amscd}
\usepackage{blkarray}
\usepackage{mathbbol}
\usepackage{setspace}
\usepackage[inner=2.5cm,outer=2.5cm, bottom=3.3cm]{geometry}
\usepackage{tikz, tikz-cd}

\usetikzlibrary{matrix}
\usetikzlibrary{arrows}

\BibSpec{collection.article}{%
	+{}  {\PrintAuthors}                {author}
	+{,} { \textit}                     {title}
	+{.} { }                            {part}
	+{:} { \textit}                     {subtitle}
	+{,} { \PrintContributions}         {contribution}
	+{,} { \PrintConference}            {conference}
	+{}  {\PrintBook}                   {book}
	+{,} { }                            {booktitle}
	+{,} { }                            {series}
	+{, vol.} { }                            {volume}
	+{,} { }                            {publisher}
	+{,} { \PrintDateB}                 {date}
	+{,} { pp.~}                        {pages}
	+{,} { }                            {status}
	+{,} { \PrintDOI}                   {doi}
	+{,} { available at \eprint}        {eprint}
	+{}  { \parenthesize}               {language}
	+{}  { \PrintTranslation}           {translation}
	+{;} { \PrintReprint}               {reprint}
	+{.} { }                            {note}
	+{.} {}                             {transition}
	+{}  {\SentenceSpace \PrintReviews} {review}
}

\newcommand{\kk}{\mathbb{k}}

\newcommand{\bbF}{\mathbb{F}}

\newcommand{\GG}{\mathcal{G}}

\newcommand{\M}{\mathbb{M}}
\newcommand{\NN}{\normalfont\mathbb{N}}
\newcommand{\ZZ}{{\normalfont\mathbb{Z}}}
\newcommand{\Z}{{\normalfont\mathbb{Z}}}

\newcommand{\PP}{\normalfont\mathbb{P}}

\newcommand{\mm}{{\normalfont\mathfrak{m}}}

\newcommand{\SSS}{\mathbb{S}}

\newcommand{\pp}{{\normalfont\mathfrak{p}}}

\newcommand{\qqq}{\normalfont\mathfrak{q}}

\newcommand{\nnn}{\pp}

\newcommand{\rank}{\normalfont\text{rank}}
\newcommand{\reg}{\normalfont\text{reg}}

\newcommand{\depth}{\normalfont\text{depth}}

\newcommand{\Tor}{\normalfont\text{Tor}}
\newcommand{\Ext}{\normalfont\text{Ext}}

\newcommand{\Ker}{\normalfont\text{Ker}}
\newcommand{\Coker}{\normalfont\text{Coker}}
\newcommand{\Quot}{\normalfont\text{Quot}}
\newcommand{\IM}{\normalfont\text{Im}}

\newcommand{\Sym}{\normalfont\text{Sym}}
\newcommand{\Rees}{\mathcal{R}}
\newcommand{\Hom}{\normalfont\text{Hom}}
\newcommand{\grHom}{{}^*\Hom}

\newcommand{\BBB}{\mathfrak{B}}

\newcommand{\LL}{\mathbb{L}}

\newcommand{\Cc}{\mathcal{C}}

\newcommand{\HL}{\normalfont\text{H}_{\mm}}
\newcommand{\HH}{\normalfont\text{H}}
\newcommand{\bBB}{\normalfont\text{B}}
\newcommand{\zZZ}{\normalfont\text{Z}}
\newcommand{\cCC}{\normalfont\text{C}}

\newcommand{\AAA}{\mathfrak{A}}

\newcommand{\bideg}{\normalfont\text{bideg}}
\newcommand{\Proj}{\normalfont\text{Proj}}
\newcommand{\Spec}{{\normalfont\text{Spec}}}
\newcommand{\spec}{{\normalfont\text{Spec}}}

\newcommand{\N}{\mathbb{N}}

\newcommand{\ip}{\mathfrak{p}}

\newtheorem{theorem}{Theorem}[section]

\newtheorem{headthm}{Theorem}

\newaliascnt{headcor}{headthm}

\aliascntresetthe{headcor}

\newaliascnt{headconj}{headthm}

\aliascntresetthe{headconj}

\newaliascnt{corollary}{theorem}
\newtheorem{corollary}[corollary]{Corollary}
\aliascntresetthe{corollary}

\newaliascnt{lemma}{theorem}
\newtheorem{lemma}[lemma]{Lemma}
\aliascntresetthe{lemma}

\newaliascnt{conjecture}{theorem}

\aliascntresetthe{conjecture}

\newaliascnt{proposition}{theorem}
\newtheorem{proposition}[proposition]{Proposition}
\aliascntresetthe{proposition}

\theoremstyle{definition}
\newaliascnt{definition}{theorem}
\newtheorem{definition}[definition]{Definition}
\aliascntresetthe{definition}

\newaliascnt{notation}{theorem}
\newtheorem{notation}[notation]{Notation}
\aliascntresetthe{notation}

\newaliascnt{example}{theorem}
\newtheorem{example}[example]{Example}
\aliascntresetthe{example}

\newaliascnt{examples}{theorem}

\aliascntresetthe{examples}

\newaliascnt{remark}{theorem}
\newtheorem{remark}[remark]{Remark}
\aliascntresetthe{remark}

\newaliascnt{problem}{theorem}

\aliascntresetthe{problem}

\newaliascnt{construction}{theorem}

\aliascntresetthe{construction}

\newaliascnt{setup}{theorem}
\newtheorem{setup}[setup]{Setup}
\aliascntresetthe{setup}

\newaliascnt{algorithm}{theorem}

\aliascntresetthe{algorithm}

\newaliascnt{observation}{theorem}

\aliascntresetthe{observation}

\newaliascnt{defprop}{theorem}

\aliascntresetthe{defprop}

\def\equationautorefname~#1\null{(#1)\null}
\def\sectionautorefname~#1\null{Section #1\null}
\def\subsectionautorefname~#1\null{\S #1\null}

\begin{document}

\title{Generic freeness of local cohomology and  graded specialization}

\author{Marc Chardin}
\address[Chardin]{Institut de Math\'ematiques de Jussieu. UPMC, 4 place Jussieu, 75005 Paris, France}
\email{marc.chardin@imj-prg.fr}

\author{Yairon Cid-Ruiz}
\address[Cid-Ruiz]{Department of Mathematics: Algebra and Geometry, 9000 Ghent, Belgium.}
\address[Cid-Ruiz]{Max Planck Institute for Mathematics in the Sciences, Inselstra\ss e 22, 04103 Leipzig, Germany.}
\email{cidruiz@mis.mpg.de}
\urladdr{https://ycid.github.io}

\author{Aron Simis}
\address[Simis]{Departamento de Matem\'atica, CCEN, 
	Universidade Federal de Pernambuco, 
	50740-560 Recife, PE, Brazil}
\email{aron@dmat.ufpe.br}

\begin{abstract}
	The main focus is the generic freeness of local cohomology modules in a graded setting.
	The present approach takes place in a quite nonrestrictive setting, by solely assuming that the ground coefficient ring is Noetherian. Under additional assumptions, such as when the latter is reduced or a domain, the outcome turns out to be stronger. One important application of these considerations is to the specialization of rational maps and of symmetric and Rees powers of a module. 
\end{abstract}



\maketitle


\section{Introduction}

Although the actual strength of this paper has to do with generic freeness in graded local cohomology, we chose to first give an overview of one intended application to specialization theory.

\medskip

Specialization is a classical and important subject in algebraic geometry and commutative algebra. 
Its roots can be traced back to seminal work by Kronecker, Hurwitz (\cite{Hurwitz}), Krull (\cite{Krull_I,Krull_II}) and Seidenberg (\cite{Seidenberg}).   
More recent papers where specialization is used in the classical way are \cite{Trung_specialization_in_german}, \cite{TRUNG_SPECIALIZATION}, \cite{Trung_specialization_local} and, in a different vein,  \cite{EISENBUD_HUNEKE_SPECIALIZATION}, \cite{SIMIS_ULRICH_SPECIALIZATION}, \cite{Residual_int}, \cite{Generic_residual_int}, \cite{Ulrich_RedNo},  \cite{ram1}, \cite{SPECIALIZATION_RAT_MAPS}.

In the classical setting it  reads as follows.
Let $\kk$ be a field and $R$ be a polynomial ring $R=\kk(\mathbf{z})[\mathbf{x}]=\kk(\mathbf{z})[x_1,\ldots,x_r]$ over a purely transcendental field extension $\kk(\mathbf{z})=\kk(z_1,\ldots,z_m)$ of $\kk$.
Let $I \subset R$ be an ideal of $R$ and $\alpha = (\alpha_1,\ldots,\alpha_m) \in \kk^m$. 
The specialization of $I$ with respect to the substitutions $z_i \rightarrow \alpha_i$ is given by the ideal 
$$
I_\alpha := \big\lbrace f(\alpha, \mathbf{x}) \mid f(\mathbf{z},\mathbf{x}) \in I \cap \kk[\mathbf{z}][\mathbf{x}] \big\rbrace.
$$
Setting $J: = I \cap \kk[\mathbf{z}][\mathbf{x}]$, the canonical map $\pi_\alpha : \kk[\mathbf{z}][\mathbf{x}] \twoheadrightarrow \kk[\mathbf{z}][\mathbf{x}]/\left(\mathbf{z}-\alpha\right)$ yields the identification 
\begin{equation}
\label{eq_intro_map_specialization}
I_\alpha \simeq \pi_\alpha(J),
\end{equation}
which is the gist of the classical approach.

\medskip

One aim of this paper is to introduce a notion of specialization on more general settings in the graded category, whereby 
$\kk(\mathbf{z})$ will be replaced by a Noetherian reduced ring $A$ and a finitely generated graded $A$-algebra will take the place of $R$.
The emphasis of this paper will be on graded modules, and more specifically, on the graded parts of local cohomology modules.

In order to recover the essential idea behind \autoref{eq_intro_map_specialization} in our setting, we now explain the notion of specialization used in this work. 
 In the simplest case, let $A$ be a Noetherian ring, let $R$ be a finitely generated positively graded $A$-algebra and let $\mm \subset R$ denote the graded irrelevant ideal $\mm = \left[R\right]_+$.
Here, for simplicity, let $M$ be a finitely generated torsion-free graded $R$-module having rank, with a fixed embedding $\iota : M \hookrightarrow F$ into a finitely generated graded free $R$-module $F$.
For any $\pp \in \Spec(A)$,  the specialization of $M$ with respect to $\pp$ will be defined to be
$$
\SSS_\nnn(M) = \IM\Big( \iota \otimes_A k(\nnn): M \otimes_A k(\nnn) \rightarrow F \otimes_{A} k(\pp)\Big)
$$
where $k(\pp) = A_\pp/\pp A_\pp$ is the residue field of $\pp$.
If $M$ is not torsion-free, one kills its $R$-torsion and proceed as above.
It can be shown that the definition is independent on the chosen embedding for general choice of $\pp\in\Spec(A)$ (see \autoref{lem_specialization_powers_fiber}).

The true impact of the present approach is the following.

\begin{headthm}[\autoref{cor_specialization_module}]
	Let $A$ be a Noetherian reduced ring and $R$ be a finitely generated positively graded $A$-algebra.
	Let $M$ be a finitely generated graded $R$-module having rank.
	Then, there exists a dense open subset $\mathcal{U} \subset \Spec(A)$ such that, for all $i \ge 0, j \in \ZZ$, the function
	\begin{equation*}
	\Spec(A) \longrightarrow \ZZ,  \qquad
	\nnn \in \Spec(A) \mapsto\dim_{k(\nnn)}\Bigg( {\left[\HL^i\big(\SSS_\nnn(M)\big)\right]}_{j} \Bigg)
	\end{equation*}
	is locally constant on $\mathcal{U}$.
\end{headthm}

Alas, although one can control all the graded parts  of the specialization of $M$, not so much for all higher symmetric and Rees powers, whereby the results will only be able to control certain graded strands. 

Anyway, one has enough to imply the local constancy of numerical invariants such as dimension, depth, $a$-invariant and regularity   under a general specialization (see \autoref{thm_general_fibers_powers_homolog_prop}). 

\medskip

The main tool  in this paper is the behavior of local cohomology of graded modules under generic localization with a view towards generic freeness (hence its inclusion in the title).
This is a problem of great interest in its own right, having been addressed earlier by several authors.
We will approach the matter in a quite nonrestrictive setting, by assuming at the outset that $A$ is an arbitrary Noetherian ring. 
When $A$ is a domain or, sometimes, just a reduced ring, one recovers and often extends some results by Hochster and Roberts \cite[Section 3]{HOCHSTER_ROBERTS} and by Smith \cite{KSMITH}.

An obstruction for local freeness of local cohomology of a finitely generated graded $R$-module $M$ is here described in terms of certain closed subsets of $\Spec(A)$ defined in terms of $M$ and  its $\Ext$ modules. 
To wit, it can be shown that the set
$$
U_M = \big\lbrace \pp \in \Spec(A) \mid  [M]_\mu  \otimes_A A_\ip \text{ is } A_\ip\text{-free for every } \mu \in \ZZ \big\rbrace
$$ 
is an open set of $\Spec(A)$, and that is dense when $A$ is reduced.
Its complement $T_M = \Spec(A) \setminus U_M$ will play a central role in this regard.

For convenience, set 
$
\left(M\right)^{*_A}:={^*\Hom_A}(M, A) = \bigoplus_{\nu \in \ZZ} \Hom_A\left({\left[M\right]}_{-\nu}, A\right).
$ 

The following theorem encompasses the main results in this direction, with a noted difference as to whether $A$ is a domain or just reduced.

\begin{headthm}[\autoref{thm_relative_graded_local_duality}]
	Let $A$ be a Noetherian ring and $R$ be a positively graded polynomial ring $R=A[x_1,\ldots,x_r]$ over $A$.
	Set $\delta = \deg(x_1)+\cdots + \deg(x_r) 
	\in \NN$.
	Let $M$ be a finitely generated graded $R$-module.
	\begin{enumerate}[\rm (I)]
		\item 	If $\pp \in \Spec(A)\setminus \left(T_M \cup \bigcup_{j=0}^\infty T_{\Ext_R^j(M, R)}\right)$, then the following statements hold for any $0\le i \le r$:
		\begin{enumerate}[\rm (a)]
			\item $\left[\HL^i(M \otimes_{A} A_\pp)\right]_\nu$ is free over $A_\pp$ for all $\nu \in \ZZ$.
			\item 
			For any $A_\pp$-module $N$, the natural map 
			$
			\HL^i(M) \otimes_{A_\ip} N \rightarrow \HL^i(M \otimes_{A_\ip} N)
			$
			is an isomorphism.
			\item 
			For any $A_\pp$-module $N$, there is an isomorphism 
			$$
			\HL^i\left(M \otimes_{A_\ip} N\right) \simeq {\left(\Ext_{R\otimes_{A} A_\pp}^{r-i}\left(M\otimes_{A} A_\pp, R(-\delta) \otimes_{A} A_\pp\right)\right)}^{*_{A_\pp}} \otimes_{A_\ip} N.
			$$			
		\end{enumerate}
	
		\item Let $F_\bullet : \cdots \rightarrow F_k \rightarrow \cdots \rightarrow F_1 \rightarrow F_0 \rightarrow 0$ be a graded free resolution of $M$ by modules of finite rank. 
		If $\pp \in \Spec(A) \setminus \left(T_M \cup T_{D_M^{r+1}} \cup  \bigcup_{j=0}^r T_{\Ext_R^j(M, R)}\right)$ where $D_M^{r+1} = \Coker\left(\Hom_R(F_r,R) \rightarrow \Hom_R(F_{r+1},R) \right)$, then the same statements as in {\rm (a), (b), (c)} of {\rm (I)} hold.
		
		\item If $A$ is reduced, then there exists an element $a \in A$ avoiding the minimal primes of $A$ such that, for any $0 \le i \le r$, the following statements hold:
		\begin{enumerate}[\rm (a)]
			\item $\HL^i(M \otimes_A A_a)$ is projective over $A_a$.
			\item For any $A_a$-module $N$, the natural map 
			$
			\HL^i(M) \otimes_{A_a} N \rightarrow \HL^i(M \otimes_{A_a} N)
			$
			is an isomorphism.
			\item For any $A_a$-module $N$, there is an isomorphism 
			$$
			\HL^i\left(M \otimes_{A_a} N\right) \simeq {\left(\Ext_{R\otimes_{A} A_a}^{r-i}\left(M\otimes_{A} A_a, R(-\delta) \otimes_{A} A_a\right)\right)}^{*_{A_a}} \otimes_{A_a} N.
			$$			
		\end{enumerate}
		\item If $A$ is a domain, then there exists an 	element $0 \neq a \in A$ such that $\left[\HL^i(M \otimes_A A_a)\right]_\nu$ is free over $A_a$ for all $\nu \in \ZZ$.
	\end{enumerate}
\end{headthm}

Above, the fact that when $A$ is a reduced ring, but not a domain, the local cohomology modules $\HL^i(M \otimes_A A_a)$ are projective  but not necessarily free over $A_a$, seems to be a hard knuckle.
As if to single out this difficulty, we note that
\cite[Corollary 1.3]{KSMITH} might not be altogether correct -- we give a counter-example in \autoref{examp_generic_proj}.

\medskip

Having briefly accounted for the main results on the generic localization of local cohomology and how it affects the problem of specialization, we turn to the question of specializing the powers of a module.
We consider both symmetric powers as Rees powers, leaving out wedge powers for future consideration.
Then, one is naturally led to focus on a bigraded setting as treated in \autoref{section_loc_cohom_bigrad}.
In this setting, we will be able to control certain graded strands, but unfortunately not all graded parts. This impairment is not due to insufficient work, rather mother nature as  has been proved by Katzman (\cite{KATZAMAN}; see \autoref{examp_Katzman}).

For the main result of the section, one assumes that $A$ is a Noetherian reduced ring and $\mathfrak{R} = A[x_1,\ldots,x_r,y_1,\ldots,y_s]$ is a
 bigraded polynomial ring with $\bideg(x)=(\delta_i,0)$ with $\delta_i > 0$ and $\deg(y_i)=(-\gamma_i,1)$ with $\gamma_i \ge 0$.
Set $\mm:=(x_1,\ldots,x_r)\mathfrak{R}$, the extension to  $\mathfrak{R}$ of the irrelevant graded ideal of the positively graded polynomial ring $A[x_1,\ldots,x_r]$.

Then:
\begin{headthm}
{\rm (\autoref{thm_local_cohom_general_fiber_bigrad_mod})}
Let $\M$ be a finitely generated bigraded $\mathfrak{R}$-module. For a fixed integer $j \in \ZZ$, there exists a dense open subset $\mathcal{U}_j \subset \Spec(A)$ such that, for all $i \ge 0, \nu \in \ZZ$, the function
\begin{equation*}
\Spec(A) \longrightarrow \ZZ, \qquad  
\nnn \in \Spec(A) \mapsto\dim_{k(\nnn)}\Bigg( {\left[\HL^i\big(\M \otimes_A k(\nnn)\big)\right]}_{(j,\nu)} \Bigg)
\end{equation*}
is locally constant on $\mathcal{U}_j$.	
\end{headthm}

The rationale of the paper is that the first four sections deal with the algebraic tools regarding exactness of  fibered complexes, local cohomology of general fibers and generic freeness of graded local cohomology, whereas the last section
\autoref{section_specialization_powers_mods} contains the applications of the main theorems so far to various events of specialization.
For the sake of visibility, we organized this section in three subsections, each about the specialization of objects of different nature, so to say.
Thus, the first piece concerns as to how the local cohomology of the specialized powers of a graded module behaves. 
To skip the technical preliminaries, we refer the reader directly to the corresponding results \autoref{thm_general_fibers_symmetric_powers} and \autoref{thm_general_fibers_powers}.
Some of these should be compared with the results of \cite{TRUNG_SPECIALIZATION}, though the techniques are different. 

The second part of \autoref{section_specialization_powers_mods} concerns the problem of  specializing rational maps, pretty much in the spirit of the recent paper  
\cite{SPECIALIZATION_RAT_MAPS}.
Namely,  one gives an encore of the fact, previously shown in loc.~cit.,  that the (topological) degree of the rational map and the degree (multiplicity) of the corresponding image remain stable under a general specialization of the coefficients involved in the given data.
Here, the outcome shapes up as a consequence of \autoref{thm_general_fibers_powers} and \cite[Corollary 2.12]{MULTPROJ}.
Given the known relations between rational maps and both the saturated special fiber rings and the $j$-multiplicities, it seems only natural to consider the latter under general specialization.

The last part is a short account of typical numerical module invariants, such as dimension, depth, $a$-invariant and regularity,  showing that they are locally constant when tensoring with a general fiber and under a general specialization. 
Taking a more geometric view one shows that, for a coherent sheaf,  the dimension of the cohomology of a general fiber is locally constant for every twisting of the sheaf, which  can be looked upon as a slight improvement on the well-known upper semi-continuity theorem.

\section{Exactness of the fibers of a complex}
\label{section_exactness}

In this section one studies how the process of taking tensor product with a fiber affects the homology of a complex.
In the main result of the section one shows that, under a nearly unrestrictive setting, an exact complex remains exact after taking tensor product with a general fiber.
This result can be seen as a vast generalization and an adaptation of \cite[Theorem 1.5, Proposition 2.7]{TRUNG_SPECIALIZATION}.

Since one is interested in certain naturally bigraded algebras -- such as the symmetric or the Rees algebra of graded modules -- and there is no significant difference between a bigraded setting or a general graded one, one will deal with the following encompassing setting.

\begin{setup}
	\label{setup_mult_grad_setting}
	Let $A$ be a ring -- always assumed to be commutative and unitary. 
	Let $R$ be the polynomial ring $R=A[x_1,\ldots,x_r]$ graded by an Abelian group $G$.
	Assume that $\deg(a)=0 \in G$ for $a \in A$ and that there is a $\ZZ$-linear map $\psi : G \rightarrow \ZZ$ such that $\psi\left(\deg(x_i)\right) > 0$ for all $1 \le i \le r$.
\end{setup}

Under the above setup, which is assumed throughout, for any finitely generated graded $R$-module $M$, the graded strands $[M]_\mu$ ($\mu\in G$) are finitely generated $A$-modules. 

The notation below will be used throughout the paper.

\begin{notation}
	For a complex of $A$-modules
	$
	P_\bullet: \, \cdots  \xrightarrow{\phi_{i+1}} P_i \xrightarrow{\phi_i}  \cdots \xrightarrow{\phi_2} P_1 \xrightarrow{\phi_1} P_0,
	$
	one sets
	$\zZZ_i\left(P_\bullet\right) := \Ker(\phi_i)$,	$\bBB_i\left(P_\bullet\right) := \IM(\phi_{i+1})$,
	$\HH_i\left(P_\bullet\right) := \zZZ_i(P_\bullet)/\bBB_i(P_\bullet)$,	and $\cCC_i\left(P_\bullet\right) := P_i/\bBB_i(P_\bullet) \,\supset\, \HH_i\left(P_\bullet\right)$
	for all $i \in \ZZ$.
\end{notation}

\begin{remark}
	\label{rem_eq_dim_in_terms_of_Cokers}
	One of the few general assertions at this point is the following:
	for a complex of $A$-modules $P_\bullet$ and an $A$-module $N$, one has a four-term exact sequence
	\begin{align*}	
	0 \rightarrow \HH_i\big(P_\bullet \otimes_A N\big) \rightarrow \cCC_i(P_\bullet) \otimes_A N \rightarrow 
	P_{i-1} \otimes_A N \rightarrow \cCC_{i-1}(P_\bullet) \otimes_A N \rightarrow 0
	\end{align*}
	of $A$-modules.
\end{remark}

\begin{lemma}\label{finiteshifts}	
	Under {\rm \autoref{setup_mult_grad_setting}}, let $F_\bullet$ be a graded complex of finitely generated free $R$-modules. 
	Then, for every integer $i$, there exists a finite set $D(i)\subseteq G$ such that for any homomorphism $\phi :A\rightarrow \kk$ from $A$ to a field $\kk$, the shifts in the minimal free graded resolution of $\HH_i (F_\bullet \otimes_A \kk)$ belong to $D(i)$. 
\end{lemma}
\begin{proof}
	
	For the sake of clarity we divide the proof into three shorter steps.
	
	\smallskip
	
	{\sc Step 1.}
	First,  assume that $G=\Z$ and $\deg (x_i )=1$.
	From \cite[Lemma 2.2 (2)]{CHH} applied to the complex $F_\bullet \otimes_{A} \kk$ of finitely generated free $R \otimes_{A} \kk (\simeq \kk[x_1,\ldots,x_r])$-modules, we obtain a constant $C(i)$ which is independent of the field $\kk$ and that bounds the Castelnuovo-Mumford regularity of $\HH_i\left(F_\bullet \otimes_{A} \kk\right)$, i.e., 
	$\reg\left(\HH_i\left(F_\bullet \otimes_{A} \kk\right)\right) \le C(i)$.
	Therefore, by using the definition of Castelnuovo-Mumford regularity in terms of minimal free resolutions, we get the existence of such finite set $D(i)$ (that is independent of $\kk$).

	\smallskip

	{\sc Step 2.} Next,  assume that $G$ is generated by the elements $\deg(x_i) \in G$.
	Let $g_i:=\deg(x_i) \in G$ and $d_i:=\psi (g_i) \in \ZZ_{>0}$.
	If $f\in R$ is $G$-homogeneous, it is also $\Z$-homogeneous for the induced $\Z$-grading  $\deg_\Z (F):=\psi (\deg_G (F))$. 
	Then, since the degrees  of $\HH_i(F_\bullet\otimes_{A} \kk)$ (in the $\ZZ$-grading induced by $G$) are bounded below by the smallest shift in $F_i$ (in the $\ZZ$-grading induced by $G$), the condition $\deg_\Z(x_i)>0$ guarantees the existence of a finite set $D_\mu(i) \subset G$ (independent of $\kk$) such that 
	$$
	\Big\lbrace \nu \in G \mid \left[\Tor_{R\otimes_{A} \kk}^j(\HH_i(F_\bullet\otimes_{A} \kk),\kk)\right]_\nu \neq 0 \text{ for some } j \ge 0\Big\rbrace \cap \psi ^{-1}(\mu )\subseteq D_\mu(i)
	$$  for every $\mu\in\Z$. 

	It is then sufficient to prove the result for the $\Z$-grading induced by $G$. 

	Set $S:=R[x_0]$, where $x_0$ is a new variable. 
	Then $S$ has two gradings, one is the standard grading as a polynomial ring over $A$, the other comes as an extension of the $\Z$-grading of $R$ by setting $\deg_\ZZ(x_0)=0$.
	Now, given a $G$-homogeneous element $f=\sum_\alpha c_\alpha x^\alpha \in R$, consider the polynomial 
	$$
	f':=\sum_\alpha c_\alpha x_0^{\deg_\Z (x^\alpha )-\vert \alpha\vert}x^\alpha \in S.
	$$
	The latter is homogeneous in the above two gradings of $S$.  
	Set
	$$
	\bideg (f'):=(\deg(f'), \deg_\Z(f'))=\deg_\Z(f) \cdot (1,1)
	$$ 
	where one uses the degrees corresponding to the above two gradings. 

	Likewise, given a matrix  $M=(f_{i,j})$ of $G$-homogeneous elements, one sets  $M':=(f'_{i,j})$. 

	Now, as $\bideg (f')=\deg_\Z (f) \cdot(1,1)$ for any $G$-homogeneous element $f$, homogenizing all the maps in $F_\bullet$ in this way provides a complex $F'_\bullet$ of standard graded free $S$-modules relative to the first component of the grading,  with shifts controlled in terms of the initial ones. 
	It then follows from the standard graded case (treated in {\sc Step 1}) that the minimal bigraded free $(S \otimes_{A} \kk)$-resolution of $\HH_i(F'_\bullet \otimes_{A} \kk)$ has shifts $(a,b)$ with $a$ bounded above by an integer $K(i)$ that does not depend on the field $\kk$. 
	Specializing $x_0$ to $1$ provides a (possibly non minimal) $\deg_\Z$-graded free $(R\otimes_{A} \kk)$-resolution of $\HH_i(F_\bullet \otimes_{A} \kk)$ (see, e.g., \cite[proof of Corollary 19.8]{EISEN_COMM}). 
	But for any monomial $\deg_\Z (x^\alpha )\leq \max_j\{ d_j\} \deg (x^\alpha )$. 
	Hence, all shifts in the minimal free  $(R \otimes_{A} \kk)$-resolution of $\HH_i(F_\bullet \otimes_{A} \kk)$ are bounded above by $K(i)\max_j\{ d_j\}$, and the claim follows, since the initial degree of  $\HH_i(F_\bullet \otimes_{A} \kk )$ is bounded below by the smallest shift in $F_i$. 

\smallskip

	{\sc Step 3.}
	Finally, let $G'$ be the subgroup of $G$ generated by the $g_i$'s. 
	If $h_1,\ldots h_s$ are representatives of the different classes modulo $G'$ of the shifts appearing in $F_{i-1},F_i$ and $F_{i+1}$, the $(R\otimes_{A} \kk)$-module $\HH_i(F_\bullet \otimes_{A} \kk)$ is the direct sum of the homology of the strands corresponding to summands whose shifts belong to these classes. 
	Again, each of these gives rise, by the above proof, to only finitely many options for the shifts in the minimal free resolution of the corresponding strand of $\HH_i(F_\bullet \otimes_{A} \kk)$.
\end{proof}

The gist of \autoref{finiteshifts} is the ability of reducing the vanishing of the fiber homology of a free graded complex of $R$-modules to a finite number of degrees. 
This will be transparent in the following result.

Recall the usual notation by which, for any $\nnn \in \Spec(A)$, $k(\nnn)$ denotes the residue field 
$$
k(\nnn):=A_\nnn/\nnn A_\nnn = \Quot(A/\nnn).
$$

\begin{lemma}\label{TM}
	Under {\rm \autoref{setup_mult_grad_setting}},
	let $F_\bullet$ be a graded complex of finitely generated free $R$-modules.\smallskip
	
	\begin{enumerate}[\rm (i)]
		\item For every $i$, there exists a finite set of degrees $D(i)$ such that, for any prime ideal $\ip \in \spec (A)$, the following are equivalent:
		\begin{enumerate}[\rm (a)]
			\item  $\HH_i (\left[F_\bullet\right]_\mu \otimes_A k(\ip )) =0$, for every $\mu\in D(i)$, 
			
			\item  $\HH_i (F_\bullet \otimes_A k(\ip ))=0$.
		\end{enumerate}
	
		\item For every $i$, the set $\big\lbrace \ip \in \spec (A) \mid \HH_i (F_\bullet \otimes_A k(\ip )) = 0 \big\rbrace$ is open in $\spec (A)$.
		\item  Assume that $A$ is locally Noetherian, $F_i=0$ for $i<0$ and $\HH_i(F_\bullet )=0$ for $i>0$. Set $M :=\HH_0 (F_\bullet )$. 
		Then, the set $\{\ip\in \spec (A)\mid [M]_\mu  \otimes_A A_\ip \mbox{is $A_\ip$-free for all $\mu \in G$}\}$ is open in $\spec (A)$.
	\end{enumerate}
\end{lemma}

\begin{proof}
	
	(i) It is clear that (b) implies (a). 
	For the converse, it follows from \autoref{finiteshifts} that $\left[\HH_i ((F_\bullet) \otimes_A k(\ip ))\right]_\mu$ is generated by elements whose $G$-degree belong to a finite set $D(i)\subseteq G$. 
	In particular, if $\left[\HH_i ((F_\bullet) \otimes_A k(\ip ))\right]_\mu$ vanishes for $\mu\in D(i)$, it follows that $\HH_i (F_\bullet \otimes_A k(\ip ))=0$. 
	
	(ii) For a given $\mu$,  $\left[\HH_i ((F_\bullet) \otimes_A k(\ip ))\right]_\mu \not= 0$ is equivalent to the condition $\rank [(d_{i+1}\otimes _A k(\ip ))]_\mu +\rank [(d_{i}\otimes _A k(\ip ))]_\mu <\rank [F_i]_\mu$, a closed condition in terms of ideals of minors of matrices representing these graded pieces of the differentials $d_{i+1}$ and $d_{i}$ of $F_\bullet$.
	So, the result follows from part (i).
	
	(iii) By the local criterion for flatness (see, e.g., \cite[Theorem 6.8]{EISEN_COMM}), for any $\mu$, the following are equivalent:
	
	(a)$_\mu$  $\Tor_1^{A_\ip}([M]_\mu  \otimes_A A_\ip , k(\ip ))=[\HH_1 (F_\bullet \otimes_A k(\ip ))]_\mu = 0$,
	
	(b)$_\mu$\; $[M]_\mu   \otimes_A A_\ip$ is $A_\ip$-flat,
	
	(c)$_\mu$\; $[M]_\mu  \otimes_A A_\ip$ is $A_\ip$-free,\\
	where the last three conditions coincide since $[M]_\mu  \otimes_A A_\ip$ is a finitely presented $A_\ip$-module for any $\mu$. 
	So, the conclusion follows from part (ii).
\end{proof}

\begin{notation}
	\label{not_TM}
	 For any finitely generated graded $R$-module $M$, one denotes by $T_M$ the complement in $\Spec(A)$  of the open set 
	 $$U_M := \{\ip\in \spec (A)\mid [M]_\mu  \otimes_A A_\ip \mbox{is $A_\ip$-free for all $\mu \in G$}\}
	 $$
	 introduced in \autoref{TM}(iii).
\end{notation}

\begin{remark}
\label{rem_dense_reduced}
	 When $A$ is Noetherian and $M$ is a finitely generated graded $R$-module, $T_M$ is a closed subset of $\Spec(A)$ by \autoref{TM}(iii).
Furthermore, if $\ip$ is a minimal prime of $A$ such that $A_\ip$ is a field, then $\ip \not\in T_M$.
	In particular, if $A$ is generically reduced then $U_M$ is dense in $\Spec(A)$. 
	This is in particular the case when $A$ is reduced -- a frequent assumption in this paper.	
\end{remark}

\begin{lemma}\label{flatexchange}
	Let $P_\bullet$ be a complex of $A$-modules and let $s\geq 0$ denote an integer. Assume that:
	\begin{enumerate}
			\item[{\rm (a)}]  $P_i$ is  $A$-flat for every $0\leq i\leq s$.
			\item[{\rm (b)}]  $\HH_i(P_\bullet )$ is  $A$-flat for every $0\leq i\leq s$.
	\end{enumerate}
	Then, for any $A$-module $N$ and any $0 \le i \le s$, one has that 
	$$
	\HH_i(P_\bullet \otimes_A N)\simeq \HH_i(P_\bullet )\otimes_A N.
	$$
\end{lemma}
\begin{proof}
	 Let $F_\bullet$ be a free $A$-resolution of $N$. The two spectral sequences associated to the double complex with components $P_p\otimes_A F_q$ have respective second terms $\Tor_q^A (\HH_p(P_\bullet ),N)$ and $\HH_p(\Tor_q^A (P_\bullet , N))$. 
	 As $\Tor_q^A(P_\bullet ,N)=0$ for $q>0$ by (a) and $\Tor_q^A (\HH_p(P_\bullet ),N)$ for $q>0$ by (b), the statement follows.
\end{proof}

For a graded $R$-module  $M$, denote for brevity
$$
\left(M\right)^{*_A}={}^*\Hom_A(M, A): = \bigoplus_{\nu \in G} \Hom_A\left({\left[M\right]}_{-\nu}, A\right).
$$ 
Note that $\left(M\right)^{*_A}$ has a natural structure of graded $R$-module.

\begin{lemma}
	\label{lem_graded_dual_cocomp}
	Let $P^\bullet$ be a co-complex of finitely generated graded $R$-modules. Assume that:
	\begin{enumerate}[\rm (a)]
		\item  $P^i$ is $A$-flat for all $i \ge 0$.
		\item  $\HH^i(P^\bullet )$ is  $A$-flat for all $i \ge 0$.
		\item $A$ is Noetherian.
	\end{enumerate}
	Then, for all $i \ge 0$, one has that
	$$
	\HH_i\big(\left(P^\bullet\right)^{*_A} \big)\simeq \left(\HH^i(P^\bullet)\right)^{*_A}.
	$$	
\end{lemma}
\begin{proof}
	First notice that, since $\left[P^p\right]_{\mu}$ is finitely generated over $A$ for any $\mu$ and $A$ is Noetherian, the modules $P^p$, $(P^p)^{*_A}$, $\HH^p(P^\bullet)$ and $\left(\HH^p(P^\bullet)\right)^{*_A}$ are direct sums of finitely presented $A$-modules. 
	Hence, each one of these is $A$-flat if and only if is $A$-projective.
	
	Let $I^\bullet$ be an injective $A$-resolution of $A$. The two spectral sequences associated to the second quadrant double complex $\mathbb{G}^{\bullet,\bullet}$ with components $\mathbb{G}^{-p,q} = \bigoplus_{\mu \in G} \Hom_A ([P^p]_{-\mu },I^q)$ have respective second terms 
	$$
	{}^\text{II}E_2^{p,-q} = \bigoplus_{\mu \in G} \Ext^p_A \big([\HH^q(P^\bullet)]_{-\mu},A\big) \quad\text{ and }\quad {}^\text{I}E_2^{-p,q} = \bigoplus_{\mu \in G} \HH^{-p}\big(\Ext^q_A ([P^\bullet]_{-\mu}, A)\big).
	$$  
	From (a) and (b) we obtain that $\Ext_A^q([P^\bullet]_{-\mu}, A) = 0$ for $q > 0$ and $\Ext_A^p([\HH^q(P^\bullet)]_{-\mu}, A) = 0$ for $p > 0$, respectively. 
	Therefore
	$$
	\HH_p\big(\left(P^\bullet\right)^{*_A} \big) = \HH^{-p}\big(\left(P^\bullet\right)^{*_A} \big) \simeq \left(\HH^p(P^\bullet)\right)^{*_A},\; \forall p\geq 0,
	$$
	and so the result follows.
\end{proof}

The following local version of the classical generic freeness lemma will be used over and over.

\begin{corollary}
	\label{cor_generic_projective}
Under {\rm \autoref{setup_mult_grad_setting}}, let $A$ be a reduced  Noetherian ring and let $M$ be a finitely generated graded $R$-module.
Then, there exists an element $a \in A$ avoiding the minimal primes of $A$ such that $M_a$ is a projective $A_a$-module.
\end{corollary}
	\begin{proof}
		From \autoref{TM}(iii) and the prime avoidance lemma one can find $a \in A$ avoiding the minimal primes of $A$ such that $D(a) \subset \Spec(A) \setminus T_M$.
	\end{proof}

In the sequel, given $\pp \in \Spec(A)$, an $R$-module $M$ and an $A_\pp$-module $N$, the $(R\otimes_{A} A_\pp)$-module $M_{\pp}\otimes_{A_\ip} N=\left(M \otimes_A A_\pp \right)\otimes_{A_\ip} N$ will as usual be denoted by $M \otimes_{A_\ip} N$.
By the same token, given $a \in A$, $M \otimes_{A_a} N$ will denote $M_a \otimes_{A_a} N$.

Next is the main result of the section. 
For a complex of finitely generated graded $R$-modules, the following theorem gives an explicit closed subset of  $\Spec(A)$ outside which  homology commutes with tensor product.

\begin{theorem}
	\label{thm_general_specialization_complexes}
	Under {\rm \autoref{setup_mult_grad_setting}}, let $A$ be a Noetherian ring and	let $P_\bullet$ be a complex of finitely generated graded $R$-modules with $P_i=0$ for $i<0$. 
	 Given an integer $s\geq 0$ and a prime  $\ip\in\Spec(A)\setminus \bigcup_{i=0}^{s}\left(T_{P_i}\cup T_{\HH_i(P_\bullet )}\right)$, then 
	$$
	\HH_i (P_\bullet \otimes_{A_\ip }N) \simeq  \HH_i (P_\bullet )\otimes_{A_\ip }N,
	$$
	for any $A_\ip$-module $N$ and for every $0\leq i\leq s$.	
	\begin{proof}
		Let $P_{\bullet\bullet}$ be a Cartan--Eilenberg graded free $R$-resolution of $P_\bullet$ with finitely generated summands. 
		The totalization $T_\bullet$ of $P_{\bullet\bullet}$ is a complex of finitely generated graded free $R$-modules with  $\HH_i (T_\bullet ) = \HH_i (P_\bullet )$ for all $i$. 
		On the one hand,  
		$$
		\HH_i (T_\bullet \otimes_{A_\ip }N) \simeq \HH_i (T_\bullet )\otimes_{A_\ip }N = \HH_i (P_\bullet )\otimes_{A_\ip }N, \ \forall \ 0\leq i\leq s,
		$$
		by \autoref{TM}(iii) and \autoref{flatexchange} since $\ip\not\in \bigcup_{i=0}^{s}T_{\HH_i(T_\bullet )}=\bigcup_{i=0}^{s}T_{\HH_i(P_\bullet )}$. 
		
		On the other hand, the spectral sequence from $P_{\bullet\bullet}\otimes_{A_\ip} N$ with first term $E^1_{p,q}=\Tor_q^{A_\ip}(P_p  \otimes_A A_\ip , N)$, shows that 
		$$
		\HH_i (T_\bullet \otimes_{A_\ip }N) \simeq \HH_i (P_\bullet \otimes_{A_\ip }N), \ \forall \ 0\leq i\leq s,
		$$
		by \autoref{TM}(iii) since $\ip\not\in \cup_{i=0}^{s}T_{P_i}$. 		
	\end{proof}
\end{theorem}

\section{Generic freeness of graded local cohomology modules}
\label{section_gen_freeness_local_cohom}

In this section one is concerned with the generic freeness of graded local cohomology modules.
Here one extends the results of \cite{KSMITH} and \cite[Section 3]{HOCHSTER_ROBERTS} to a graded environment,  adding a few generalizations.

The following setup will hold throughout the section.

\begin{setup}
	\label{setup_generic_freeness_local_cohom}
	 Keep the notation introduced in \autoref{setup_mult_grad_setting}, so that $R=A[x_1,\ldots,x_r]$ is a $G$-graded polynomial ring. Assume in addition that $A$ is Noetherian and set  $\mm=(x_1,\ldots,x_r) \subset R$ and $\delta=\deg(x_1)+\cdots+\deg(x_r) \in G$.
	 Recall that $\HL^r(R)\simeq \frac{1}{x_1\cdots x_r}A[x_1^{-1},\ldots,x_r^{-1}]$.
\end{setup}

\begin{remark}
	\label{rem_finite_gen_grad_parts_local_cohom}
	Let $M$ be a finitely generated graded $R$-module.
	Since one is assuming that $\psi(\deg(x_i))>0$, it follows that ${\left[\HL^i(M)\right]}_\nu$ is a finitely generated $A$-module for all $i \ge 0$ and $\nu \in G$ (see \cite[Theorem 2.1]{CHARDIN_POWERS_IDEALS}).
\end{remark}

Consider the canonical perfect pairing of free $A$-modules in ``top'' cohomology
$$
{\left[R\right]}_\nu \otimes_{A} {\left[\HL^r(R)\right]}_{-\delta-\nu} \rightarrow {\left[\HL^r(R)\right]}_{-\delta} \simeq A
$$
 inducing a canonical graded $R$-isomorphism 
$
	\HL^r(R) \simeq {\left(R(-\delta)\right)}^{*_A} = \grHom_A\left(R(-\delta), A\right).	
$

The functor ${\left(\bullet\right)}^{*_A}$ has been introduced in the previous section.
It can be regarded as a relative version (with respect to $A$) of the graded Matlis dual. 

\begin{lemma}
	\label{lem_props_Maltis_relative}
	 Let 
		$
		F_\bullet
		$
		be a complex of finitely generated graded free $R$-modules.
		Then, one has the isomorphism of complexes 
		$
		\HL^r(F_\bullet) \simeq {\Big(\Hom_R(F_\bullet
			, R(-\delta))\Big)}^{*_A}.
		$
	\begin{proof}
This is well-known (see, e.g., \cite[Section 2.15]{JOUANOLOU_IDEAUX_RES},  \cite[Corollary 1.4]{DUALITY_TAMENESS}).
	\end{proof}
\end{lemma}

\begin{lemma}
	\label{lem_isoms_spectral_seq}
	Let $F_\bullet$ stand for a graded free resolution of  a finitely generated graded $R$-module $M$ by modules of finite rank.
	If $M$ is $A$-flat, then  
		$$
		\HL^i(M \otimes_{A }N) \simeq \HH_{r-i}\big(\HL^r(F_\bullet) \otimes_{A} N\big)
		$$
		for any $A$-module $N$.	
\begin{proof}
		Consider  the double complex $\Cc^\bullet_{\mm}F_\bullet \otimes_{A} N$ obtained by taking the \v{C}ech complex on $F_\bullet \otimes_{A} N$.
		Since $M$ is $A$-flat, $F_\bullet \otimes_{A} N$ is acyclic and $\HH_0\left(F_\bullet \otimes_{A} N\right)\simeq M \otimes_{A} N$.
		Therefore, as localization is exact and 
		$$
		\HL^i(R\otimes_{A} N)\simeq \begin{cases}
			\HL^r(R) \otimes_{A} N \quad\, \text{ if } i = r\\
			0 \qquad\qquad\qquad\quad \text{otherwise},
		\end{cases}
		$$
		by analyzing the spectral sequences coming from the double complex $\Cc^\bullet_{\mm}F_\bullet \otimes_{A} N$, the isomorphism $\HL^i(M \otimes_{A}N) \simeq \HH_{r-i}\big(\HL^r(F_\bullet) \otimes_{A_\ip} N\big)$ follows.
	\end{proof}
\end{lemma}

For the proof of the main theorem of the section, we need a version of the celebrated Grothendieck's generic freeness lemma in  a  more encompassing  graded environment.
We take verbatim the basic assumptions of the non-graded version first stated in \cite[Lemma 8.1]{HOCHSTER_ROBERTS_0}, making the needed adjustment in the graded case. 
The standard assumption on the ring $A$ is that it be a  domain, but we also give a generic projectivity counterpart if $A$ is just assumed to be reduced. 
 
In order to state a bona fide graded version, we make the following convention.
First, $A$ is also considered as a  $G$-graded ring, with (trivial) grading $\left[A\right]_\nu = 0$ for $\nu \neq 0 \in G$.
In addition, an $A$-module ${H}$ is said to be \emph{$G$-graded} if it has a direct summands decomposition ${H} = \bigoplus_{\mu \in G} \left[{H}\right]_\nu$ indexed by $G$, where each $\left[{H}\right]_\nu$ is an $A$-module.
For a $G$-graded $A$-algebra $\BBB$ and a $G$-graded $\BBB$-module $M$, one says that a $G$-graded $A$-module $H$ is a \emph{$G$-graded $A$-submodule} of $M$ if one has $\left[{H}\right]_\nu \subseteq \left[M\right]_\nu$ for all $\nu \in G$.

\begin{theorem}
	\label{thm_generic_freeness_mult_grad}
	Assume {\rm \autoref{setup_generic_freeness_local_cohom}}.
	In addition, let $\BBB\supset R$ be a finitely generated $G$-graded $R$-algebra.
	Let $M$ be a finitely generated $G$-graded $\BBB$-module.
	Let $E$ be a finitely generated $G$-graded $R$-submodule of $M$ and $H$ be a finitely generated $G$-graded $A$-submodule of $M$.
	Set $\mathfrak{M}=M/(E+H)$, which is a $G$-graded $A$-module. 
	\begin{enumerate}[\rm (i)]
		\item If $A$ is reduced, then there is an element $a \in A$ avoiding the minimal primes of $A$ such that $\mathfrak{M}_a$ is $A_a$-projective.
		\item If $A$ is a domain, then there is an element $0 \neq a \in A$ such that each graded component 	
		$$
		{\left[\mathfrak{M}_a\right]}_\nu, \qquad \nu \in G
		$$
		of $\mathfrak{M}_a$ is $A_a$-free.	
	\end{enumerate}
\end{theorem}	
	\begin{proof}	The proof follows along the same lines of \cite[Lemma 8.1]{HOCHSTER_ROBERTS_0} (see also \cite[Theorem 24.1]{MATSUMURA}).
		When $A$ is reduced,  one draws upon \autoref{cor_generic_projective} in order to start an appropriate inductive argument.
	\end{proof}

The ground work having been carried through the previous results so far, 
we now collect the essential applications in  the main theorem of the section.

\begin{theorem}
	\label{thm_relative_graded_local_duality}
	Under {\rm \autoref{setup_generic_freeness_local_cohom}},
	let $M$ be a finitely generated graded $R$-module. 
	\begin{enumerate}[\rm (I)]
		\item 	If $\pp \in \Spec(A)\setminus \left(T_M \cup \bigcup_{j=0}^\infty T_{\Ext_R^j(M, R)}\right)$, then the following statements hold for any $0 \le i \le r$:
		
		\begin{enumerate}[\rm (a)]
			\item $\left[\HL^i(M \otimes_{A} A_\pp)\right]_\nu$ is free over $A_\pp$ for all $\nu \in G$.
			\item 
			 For any $A_\pp$-module $N$, the natural map 
			$
			\HL^i(M) \otimes_{A_\ip} N \rightarrow \HL^i(M \otimes_{A_\ip} N)
			$
			is an isomorphism.
			\item 
			For any $A_\pp$-module $N$, there is an isomorphism 
			$$
			\HL^i\left(M \otimes_{A_\ip} N\right) \simeq {\left(\Ext_{R\otimes_{A} A_\pp}^{r-i}\left(M\otimes_{A} A_\pp, R(-\delta) \otimes_{A} A_\pp\right)\right)}^{*_{A_\pp}} \otimes_{A_\ip} N.
			$$			
		\end{enumerate}
	
		\item Let $F_\bullet : \cdots \rightarrow F_k \rightarrow \cdots \rightarrow F_1 \rightarrow F_0$ be a graded free resolution of $M$ by modules of finite rank. 
		If $\pp \in \Spec(A) \setminus \left(T_M \cup T_{D_M^{r+1}} \cup  \bigcup_{j=0}^r T_{\Ext_R^j(M, R)}\right)$ where $D_M^{r+1} = \Coker\left(\Hom_R(F_r,R) \rightarrow \Hom_R(F_{r+1},R) \right)$, then the same statements as in {\rm (a), (b), (c)} of {\rm (I)} hold.
	
		\item If $A$ is reduced, then there exists an element $a \in A$ avoiding the minimal primes of $A$ such that, for any $0 \le i \le r$, the following statements hold:
		\begin{enumerate}[\rm (a)]
			\item $\HL^i(M \otimes_A A_a)$ is projective over $A_a$.
			\item For any $A_a$-module $N$, the natural map 
			$
			\HL^i(M) \otimes_{A_a} N \rightarrow \HL^i(M \otimes_{A_a} N)
			$
			is an isomorphism.
			\item For any $A_a$-module $N$, there is an isomorphism 
			$$
			\HL^i\left(M \otimes_{A_a} N\right) \simeq {\left(\Ext_{R\otimes_{A} A_a}^{r-i}\left(M\otimes_{A} A_a, R(-\delta) \otimes_{A} A_a\right)\right)}^{*_{A_a}} \otimes_{A_a} N.
			$$			
		\end{enumerate}
	\item If $A$ is a domain, then there exists an 	element $0 \neq a \in A$ such that $\left[\HL^i(M \otimes_A A_a)\right]_\nu$ is free over $A_a$ for all $\nu \in G$.
	\end{enumerate}
	\begin{proof}		
		(I) Let $F_\bullet : \cdots \rightarrow F_k \rightarrow \cdots \rightarrow F_1 \rightarrow F_0$ be a graded free resolution of $M$ by modules of finite rank.
		
		(I)(a)
		From the fact that $\pp \not\in T_M$,  \autoref{lem_isoms_spectral_seq} and \autoref{lem_props_Maltis_relative} yield the isomorphisms 
		\begin{equation}
			\label{eq_isoms_thm_local_cohom}
			\HL^i(M \otimes_{A} A_\pp) \;\simeq\; \HH_{r-i}\left(\HL^r(F_\bullet) \otimes_{A} A_\pp\right) \;\simeq\; \HH_{r-i}\left({\Big(\Hom_R(F_\bullet, R(-\delta))\Big)}^{*_A} \otimes_{A} A_\pp\right),
		\end{equation}	
		for all $i$.
		One has that
		$$
		\HH^{r-i}\left(\Hom_{R}(F_\bullet , R(-\delta))\otimes_A A_\pp\right)  \simeq \left(\Ext_{R\otimes_{A} A_\pp}^{r-i}\left(M\otimes_{A} A_\pp, R(-\delta) \otimes_{A} A_\pp\right)\right).
		$$
		Since $\pp \not\in T_{\Ext_R^j(M, R)}$ for all $j \ge 0$,  \autoref{TM}(iii) and \autoref{lem_graded_dual_cocomp} applied to the co-complex 
		$
		\Hom_{R}(F_\bullet , R(-\delta))\otimes_A A_\pp
		$
		give the following isomorphisms
		\begin{align}
			\label{eq_local_duality_relative}
			\begin{split}
			\HL^i(M \otimes_{A} A_\pp) &\simeq \HH_{r-i}\left({\Big(\Hom_R(F_\bullet, R(-\delta))\Big)}^{*_A} \otimes_{A} A_\pp\right)\\ 
			&\simeq {\left(\Ext_{R\otimes_{A} A_\pp}^{r-i}\left(M\otimes_{A} A_\pp, R(-\delta) \otimes_{A} A_\pp\right)\right)}^{*_{A_\pp}}
			\end{split}			
		\end{align}
		for all $0 \le i \le r$.
		Again, as $\pp \not\in T_{\Ext_R^j(M, R)}$, the result is obtained from \autoref{TM}(iii).
		
		\smallskip 
		
		(I)(b) The natural map $\HH_{r-i}\left(\HL^r(F_\bullet)\right) \otimes_{A_\ip} N \rightarrow \HH_{r-i}\left(\HL^r(F_\bullet) \otimes_{A_\ip} N\right)$ is an isomorphism because each $\HH_{r-i}\left(\HL^r(F_\bullet) \otimes_{A} A_\pp\right)  \simeq \HL^i(M \otimes_{A} A_\pp)$ is a free $A_\pp$-module (from part (I)(a)) and $\HL^r(F_\bullet)$ is a complex of free $A$-modules (see \autoref{flatexchange}).
		So, the result follows from \autoref{lem_isoms_spectral_seq}.
		
		\smallskip
		
		(I)(c) From part (I)(b) and \autoref{eq_local_duality_relative} we get the isomorphisms
		$$
		\HL^i(M \otimes_{A_\ip} N) \simeq \HL^i(M) \otimes_{A_\ip} N \simeq {\left(\Ext_{R\otimes_{A} A_\pp}^{r-i}\left(M\otimes_{A} A_\pp, R(-\delta) \otimes_{A} A_\pp\right)\right)}^{*_{A_\pp}} \otimes_{A_\pp}  N, 
		$$
		which gives the result.
		
		\medskip
		
		(II) We basically repeat the same steps in the proof of (I), but instead of considering the free resolution $F_\bullet$, we now analyze the truncated complex 
		$$
		F_\bullet^{\le r+1} : \quad 0 \rightarrow  F_{r+1} \rightarrow F_r \rightarrow  \cdots \rightarrow F_{1} \rightarrow F_0.
		$$
		Note that $\HH^j(\Hom_R(F_\bullet^{\le r+1},R)) \simeq \Ext_{R}^j(M,R)$ for $0 \le j \le r$ and $\HH^{r+1}(\Hom_R(F_\bullet^{\le r+1},R)) \simeq D_M^{r+1}$.
		As $\pp \in \Spec(A) \setminus \left(T_M \cup T_{D_M^{r+1}} \cup  \bigcup_{j=0}^r T_{\Ext_R^j(M, R)}\right)$, after applying \autoref{TM}(iii) and \autoref{lem_graded_dual_cocomp} to the co-complex $\Hom(F_\bullet^{\le r+1},R) \otimes_{A} A_\pp$ and invoking \autoref{eq_isoms_thm_local_cohom} we obtain the following isomorphisms 
		\begin{align*}
		\begin{split}
		\HL^i(M \otimes_{A} A_\pp) &\simeq \HH_{r-i}\left({\Big(\Hom_R(F_\bullet, R(-\delta))\Big)}^{*_A} \otimes_{A} A_\pp\right)\\ 
		& \simeq \HH_{r-i}\left({\Big(\Hom_R(F_\bullet^{\le r+1}, R(-\delta))\Big)}^{*_A} \otimes_{A} A_\pp\right) \\
		&\simeq {\left(\Ext_{R\otimes_{A} A_\pp}^{r-i}\left(M\otimes_{A} A_\pp, R(-\delta) \otimes_{A} A_\pp\right)\right)}^{*_{A_\pp}}
		\end{split}			
		\end{align*}
		for all $0 \le i \le r$.
		Since $\pp \not\in T_{\Ext_R^j(M, R)}$ for all $0 \le j \le r$, the conclusion of (I)(a) follows from \autoref{TM}(iii).
		
		The arguments in the proofs of (I)(b) and (I)(c) only depend on the conclusion of (I)(a) and  the isomorphisms given in \autoref{eq_local_duality_relative}.
		Note that we have proved the last two results in the current part (II).
		Therefore, the conclusions of parts (I)(b) and (I)(c) also follow under the assumptions of part (II).
		
		\medskip

		(III) The proof is nearly verbatim the one of the part (II). 
		By using \autoref{cor_generic_projective}, one can choose $a \in A$ avoiding the minimal primes of $A$ such that $M_a$, $D_M^{r+1} \otimes_{A } A_a$ and $\Ext_{R}^j(M,R) \otimes_{A} A_a$ are projective $A_a$-modules for $0 \le j \le r$.
		Then, \autoref{lem_isoms_spectral_seq} and \autoref{lem_props_Maltis_relative} give the isomorphisms 
		\begin{equation*}
		\HL^i(M \otimes_{A} A_a) \;\simeq\; \HH_{r-i}\left(\HL^r(F_\bullet) \otimes_{A} A_a\right) \;\simeq\; \HH_{r-i}\left({\Big(\Hom_R(F_\bullet, R(-\delta))\Big)}^{*_A} \otimes_{A} A_a\right)
		\end{equation*}	
		for all $i$.
		Again, by applying \autoref{lem_graded_dual_cocomp} to the co-complex $\Hom(F_\bullet^{\le r+1},R) \otimes_{A} A_a$ we obtain the following isomorphism
		\begin{equation}
		\label{eq_local_duality_generic}
		\HL^i(M \otimes_{A} A_a)  
		\simeq {\left(\Ext_{R\otimes_{A} A_a}^{r-i}\left(M\otimes_{A} A_a, R(-\delta) \otimes_{A} A_a\right)\right)}^{*_{A_a}}		
		\end{equation}
		for all $0 \le i \le r$.
		So, the result of part (III)(a) also follows, and parts (III)(a) and (III)(b) are obtained from (III)(a) and \autoref{eq_local_duality_generic}.

		\medskip

		(IV) Take $a^\prime \in A$ from part (III)(c) such that the isomorphisms 
		$$
		\HL^i\left(M \otimes_{A} A_{a^\prime}\right) \simeq {\left(\Ext_{R\otimes_{A} A_{a^\prime}}^{r-i}\left(M\otimes_{A} A_{a^\prime}, R(-\delta) \otimes_{A} A_{a^\prime}\right)\right)}^{*_{A_{a^\prime}}} 
		$$
		hold, for all $0 \le i \le r$.
		
		Now, let $0 \le j \le r$. For each such $j$ apply \autoref{thm_generic_freeness_mult_grad}(ii) with $\mathfrak{B}=R$ and $\mathfrak{M}= \Ext_{R}^j(M,R)$; since there are finitely many $j$'s, there exists an $a^{\prime\prime}\neq 0$ in $A$ such that $\left[\Ext_{R}^j(M,R) \otimes_{A} A_{a^{\prime\prime}}\right]_\nu$ is a free $A_{a^{\prime\prime}}$-module for all $0 \le j \le r$ and $\nu \in G$.
		So, the result follows by setting $a=a^\prime a^{\prime\prime}$.
	\end{proof}
\end{theorem}

The theorem has an important consequence, as follows.

\begin{proposition}
	\label{cor_dimension_fiber_local_cohom}
	Under {\rm \autoref{setup_generic_freeness_local_cohom}}, assume in addition that $A$ is reduced. 
	Given a finitely generated graded $R$-module $M$,  there exists a dense open subset $\mathcal{U} \subset \Spec(A)$ such that, for all $i \ge 0, \nu \in G$, the function
	\begin{equation*}
	\Spec(A) \longrightarrow \ZZ, \qquad
	\nnn \in \Spec(A) \mapsto\dim_{k(\nnn)}\Bigg( {\left[\HL^i\big(M \otimes_A k(\nnn)\big)\right]}_{\nu} \Bigg)
	\end{equation*}
	is locally constant on $\mathcal{U}$.
	\begin{proof}
		By \autoref{thm_relative_graded_local_duality}(II), there is an element $a \in A$ avoiding the minimal primes of $A$ such that for all $i \ge 0$,
		$
		\HL^i(M) \otimes_{A_a} k(\pp) \;\simeq\; \HL^i\left(M \otimes_{A_a} k(\pp)\right)
		$
		is an isomorphism and $\left[\HL^i(M \otimes_A A_a)\right]_\nu$ is a finitely generated projective module over $A_a$ for all $\nu \in G$.
		Then,  by setting $\mathcal{U}=D(a) \subset \Spec(A)$, the result follows from the fiberwise characterization of projective modules (see \cite[Exercise 20.13]{EISEN_COMM}).
	\end{proof}
\end{proposition}

Closing the section, we thought it appropriate to provide a counter-example to the result stated in \cite[Corollary 1.3]{KSMITH}.
The  example shows that when $A$ is only reduced and not a domain,  generic freeness of the local cohomology modules $\HL^i(M)$ may fail to hold.

\begin{example}
	\label{examp_generic_proj}
	Let $\kk$ be a field and $A$ be the reduced Noetherian ring $A=\frac{\kk[t]}{\left(t(t-1)\right)}$.
	Let $R$ be the polynomial ring $R=A[x]$ and let $\mm$ be the graded irrelevant ideal $\mm=(x) \subset R$.
	\begin{enumerate}[(i)]
		\item Take $M$ as 
		$
		M=\frac{R}{\left(x, t\right)}=\frac{A}{(t)}. 
		$
		It is clear that $M = \HL^0(M)$.
		Then, for any $g \in A$ avoiding the minimal primes of $A$, $0\neq \left(\frac{A}{(t)}\right)_g$ is a projective $A_g$-module but not a free $A_g$-module.
		\item Take $M$ as 
		$
		M=\frac{R}{(t)} = \frac{A}{(t)}[x]. 
		$
		One has that 
		$
		\HL^1(M) = \frac{1}{x}\left(\frac{A}{(t)}[x^{-1}] \right).
		$
		Then, for any $g \in A$ avoiding the minimal primes of $A$,  $0\neq {\HL^1(M)}_g$ is a projective $A_g$-module but not a free $A_g$-module.
	\end{enumerate}
\end{example}

\section{Local cohomology of general fibers:  bigraded case}
\label{section_loc_cohom_bigrad}

The following setup will be used throughout the section. 

\begin{setup}\rm 
	\label{setup_mult_grad_local_cohom}
	Let $A$ be a  reduced Noetherian ring.
	Consider the $(\Z\times\Z)$-bigraded polynomial ring $\mathfrak{R} = A[x_1,\ldots,x_r,y_1,\ldots,y_s]$, where $\bideg(x_i)=(\delta_i,0)$ with $\delta_i > 0$ and $\deg(y_i)=(-\gamma_i,1)$ with $\gamma_i \ge 0$.
	Consider $\mm =(x_1,\ldots,x_r)\mathfrak{R}\subset \mathfrak{R}$ as a $(\Z\times\Z)$-bigraded ideal and recall that 
	$$
	\HL^r\left(\mathfrak{R}\right)\simeq \frac{1}{x_1\cdots x_r}A[x_1^{-1},\ldots,x_r^{-1},y_1,\ldots,y_s].
	$$ 		
	Let $S$ be the standard graded polynomial ring given by
	$$
	S := A\big[y_i \mid 1 \le i \le s \text{ and } \gamma_i = 0\big] \subset A[y_1,\ldots,y_s] \subset \mathfrak{R}.
	$$	
\end{setup}

If $\M$ is a $(\Z\times\Z)$-bigraded module over $\mathfrak{R}$, then, for any $i\geq 0$, the local cohomology module $\HL^i(\M)$ has a natural structure of bigraded $\mathfrak{R}$-module.
Also, denote by ${\left[\M\right]}_j$ the $\Z$-graded $S$-module
\begin{equation}
	\label{eq_R_grad_bigrad_convention}
	{\left[M\right]}_j = \bigoplus_{\nu\in \ZZ} {\left[M\right]}_{(j,\nu)}.
\end{equation}

\begin{remark}  
		\label{rem_finite_local_cohom}
		As a particular, but important case, take $M=\mathfrak{R}$. Let $\{ y_{i_1}, \ldots, y_{i_l} \} \subset \{ y_1, \ldots, y_s \}$ stand for the subset of variables with strictly negative $x$-degree, that is, $\bideg(y_{i_t})=(-\gamma_{i_t}, 1)$ with $-\gamma_{i_t}<0$.
		Then, for a fixed $j \in \ZZ$,  ${\left[\HL^r(\mathfrak{R})\right]}_{j}=\bigoplus_{\nu\in \ZZ} {\left[\HL^r(\mathfrak{R})\right]}_{(j,\nu)}$ is a finitely generated $\Z$-graded $S$-module with a finite set of generators given by 
		\begin{equation*}
			\Bigg\lbrace\begin{array}{c|c}
				\frac{1}{x_1^{\alpha_1}\cdots x_r^{\alpha_r}} y_{i_1}^{\beta_1}\cdots y_{i_l}^{\beta_l} & 
				\begin{array}{c}
					\alpha_1\ge 1,\ldots,\alpha_r\ge 1,\; \beta_1\ge 0,\ldots,\beta_l \ge 0,\\
					-(\alpha_1\delta_1+\cdots+\alpha_r\delta_r + \beta_1\gamma_{i_1} + \cdots +\beta_l\gamma_{i_l}) = j
				\end{array}
			\end{array}\Bigg\rbrace.
		\end{equation*}
\end{remark}

Fix the following additional notation for the section.

\begin{notation}\rm
	\label{nota_bigrad_resolution}
	Let $\M$ be a finitely generated bigraded $\mathfrak{R}$-module and choose a bigraded free resolution 
	$
	\bbF_\bullet:  \cdots \xrightarrow{\phi_2} \bbF_1 \xrightarrow{\phi_1} \bbF_0 \rightarrow \M \rightarrow 0
	$
	where each $F_i$ is a finitely generated bigraded free $\mathfrak{R}$-module.
	Let 
	$
	\LL_\bullet = \HL^r(\bbF_\bullet):  \cdots  \xrightarrow{\Psi_2} \LL_1 \xrightarrow{\Psi_1} \LL_0 
	$
	be the induced complex in local cohomology where 
	$\LL_i = \HL^r(\bbF_i)$ and  $\Psi_i = \HL^r(\phi_i) : \LL_i \rightarrow \LL_{i-1}$
	for $i \ge 1$.
\end{notation}

\begin{lemma}
	\label{prop_auxiliary_results_bigrad}
	Under {\rm \autoref{setup_mult_grad_local_cohom}} and with the above notation, the following statements hold:
	\begin{enumerate}[\rm (i)]
		\item  There is an isomorphism 
		$
		\HL^i(\M) \simeq \HH_{r-i}\big(\LL_\bullet\big)
		$
		of bigraded $\mathfrak{R}$-modules for $i \ge 0$.
		\item 
		$
		{\left[\HL^i(\M)\right]}_{j} \simeq {\left[\HH_{r-i}\big(\LL_\bullet\big)\right]}_{j}
		$
		is a finitely generated graded $S$-module for $i\geq 0$ and $j \in \ZZ$.
		\item There is a dense open subset $V \subset \Spec(A)$ such that, for every $\nnn \in V$, there is an isomorphism 
		$
		\HL^i\big(\M \otimes_A k(\nnn)\big) \simeq \HH_{r-i}\big(\LL_\bullet \otimes_A k(\nnn)\big)
		$
		of bigraded $\big(\mathfrak{R} \otimes_A k(\nnn)\big)$-modules for $i \ge 0$.
		\item Fix an integer $j \in \ZZ$.
		Then, there exists an element $a \in A$ avoiding the minimal primes of $A$ such that ${\left[{(C_i(\LL_\bullet))}_a\right]}_{j}$ is a projective module over $A_a$ for $0 \le i \le r$.
	\end{enumerate}
	
	\begin{proof}		
		For (i) and (ii) see \cite[Theorem 2.1]{CHARDIN_POWERS_IDEALS}.
		
		(iii) The argument is similar to the one in \autoref{lem_isoms_spectral_seq}.
		
		(iv) Fix $0 \le i \le r$.
		From \autoref{rem_finite_local_cohom}, one has that ${\left[C_i(\LL_\bullet)\right]}_{j}$ is a finitely generated graded $S$-module. 
		Therefore, \autoref{cor_generic_projective} yields the existence of an element $a_i \in A$ avoiding the minimal primes of $A$ such that 
		$
		{\left[{(C_i(\LL_\bullet))}_{a_i}\right]}_{j}
		$ 
		is a projective $A_{a_i}$-module.
		The required result follows by taking $a=a_0a_1\cdots a_{r}$.
	\end{proof}
\end{lemma}

Next is the main result of this section.
Its proof is very short as it is downplayed by the previous lemma and its predecessors.

\begin{theorem}
	\label{thm_local_cohom_general_fiber_bigrad_mod}
	Under {\rm \autoref{setup_mult_grad_local_cohom}},
	let $\M$ be a finitely generated bigraded $\mathfrak{R}$-module and fix an integer $j \in \ZZ$.
	Then, there exists a dense open subset $\mathcal{U}_j \subset \Spec(A)$ such that, for all $i \ge 0, \nu \in \ZZ$, the function
	\begin{equation*}
	\Spec(A) \longrightarrow \ZZ, \qquad
	\nnn \in \Spec(A) \mapsto\dim_{k(\nnn)}\Bigg( {\left[\HL^i\big(\M \otimes_A k(\nnn)\big)\right]}_{(j,\nu)} \Bigg)
	\end{equation*}
	is locally constant on $\mathcal{U}_j$.
	\begin{proof}
		By \autoref{prop_auxiliary_results_bigrad}(iii) one can choose a dense open subset $U \subset \Spec(A)$ such that 
		$
		\HL^i\big(\M \otimes_A k(\nnn)\big) \simeq \HH_{r-i}\big(\LL_\bullet   \otimes_A k(\nnn)\big).
		$
		By \autoref{rem_eq_dim_in_terms_of_Cokers},  one has an exact sequence
		$$
		0 \rightarrow \HL^i\big(\M \otimes_A k(\nnn)\big) \rightarrow \cCC_{r-i}(\LL_\bullet) \otimes_{A} k(\pp) \rightarrow  \LL_{r-i-1} \otimes_{A} k(\pp) \rightarrow \cCC_{r-i-1}(\LL_\bullet) \otimes_{A} k(\pp) \rightarrow 0
		$$
		for any $\pp \in U$.
		Therefore, the result is clear by setting $\mathcal{U}_j=U\cap V_j$ with $V_j \subset \Spec(A)$ a dense open subset as in \autoref{prop_auxiliary_results_bigrad}(iv).
	\end{proof}
\end{theorem}

The following example shows that in the current setting one can only hope to control certain graded parts, as in the result of \autoref{thm_local_cohom_general_fiber_bigrad_mod}.
It also shows the crucial importance of the choice of bigrading in \autoref{setup_mult_grad_local_cohom} for the correctness of \autoref{thm_local_cohom_general_fiber_bigrad_mod}.

\begin{example}[{\cite[Theorem 1.2]{KATZAMAN}}]
	\label{examp_Katzman}
	Since the example is slightly long, for organizational purposes, we divide it into four different parts.
	The first three parts are intended to stress strange phenomena that can happen if we weaken some of the previous settings.
	
	\medskip
	
	{\sc Part 1:} This part shows that the result of \autoref{thm_relative_graded_local_duality}(III) may fail if we consider a more general situation (this part should be compared with {\sc Part 4} below).	
	
	Let $\kk$ be a field and $R$ be the graded $\kk$-algebra
	$$
	R = \frac{\kk[s,t,x,y,u,v]}{\left(xsx^2v^2 - (t+s)xyuv+ty^2u^2\right)}
	$$
	with grading $\deg(s)=\deg(t)=\deg(x)=\deg(y)=0$ and $\deg(u)=\deg(v)=1$.
	Then, for every $d\ge 2$, one has that ${\left[\HH_{R_+}^2(R)\right]}_{-d}$ has $\tau_{d-1}$-torsion where 
	$
	\tau_{d-1} =(-1)^{d-1}(t^{d-1} +st^{d-2} +\cdots+s^{d-2}t+s^{d-1}) \in \kk[s,t].
	$
	
	By \cite[Lemma 1.1(ii)]{KATZAMAN}, it gives rise to  infinitely many irreducible homogeneous polynomials $\lbrace  p_i \in \kk[s,t] \mid i \ge 1\rbrace$ such that  $\HH_{R_+}^2(R)$ has $p_i$-torsion.
	Furthermore, from \cite[proof of Theorem 1.2]{KATZAMAN}, each $(p_i)$ yields an associated prime of $\HH_{R_+}^2(R)$ in $\kk[s,t]$.
	Therefore, one cannot find an element $0\neq a \in \kk[s,t]$ such that 
	$$
	\HH_{R_+}^2(R) \otimes_{\kk[s,t]} \kk[s,t]_a
	$$ 
	is a projective $\kk[s,t]_a$-module.
	
	\medskip
	
	{\sc Part 2:}
	In this part, we specify the current example in a bigraded setting that agrees with \autoref{setup_mult_grad_local_cohom}, and we show that \autoref{thm_local_cohom_general_fiber_bigrad_mod} cannot be extended to control all the possible graded parts.
	
	Suppose that $A = \kk[s,t]$ and that $R$ is the standard bigraded $A$-algebra 
	$$
	R = A[u,v,x,y]/\left(sx^2v^2 - (t+s)xyuv+ty^2u^2\right)
	$$ 
	with $\bideg(u)=\bideg(v)=(1,0)$ and $\bideg(x)=\bideg(y)=(0,1)$.
	If one assumes that there exists a dense open subset $\mathcal{U} \subset \Spec(A)$ such that the function 
	$$\pp \in \Spec(A) \mapsto\dim_{k(\nnn)}\left( {\left[\HH_{(u,v)}^2(R)\right]}_{(j,\nu)} \otimes_A k(\nnn) \right)$$ 
	is constant on $\mathcal{U}$ for all $j,\nu \in \ZZ$, then there exists an element $0 \neq a \in A$ such that the $A_a$-module ${\left[\HH_{(u,v)}^2(R)\right]}_{(j,\nu)} \otimes_{A} A_a$ is projective for all $j,\nu \in \ZZ$ (see \cite[Exercise 20.13]{EISEN_COMM}). 
	But, this conclusion contradicts the assertion shown in {\sc Part 1}. 
	
	\medskip
	
	{\sc Part 3:}
	In this part, we provide a bigrading that does not agree with \autoref{setup_mult_grad_local_cohom} and for which the statement of \autoref{thm_local_cohom_general_fiber_bigrad_mod} would be incorrect (this bigrading is deduced for the bigrading used in \cite[proof of Theorem 1.2]{KATZAMAN}). 
	
	First, we need to recall some details from the construction made in \cite[proof of 
	Theorem 1.2]{KATZAMAN}.
	Consider $R$ as a $\ZZ^3$-graded ring by setting  $\text{trideg}(s)=\text{trideg}(t)=(0,0,0)$, $\text{trideg}(x) =\text{trideg}(y) = (1,1,0)$ and  $\text{trideg}(u)=\text{trideg}(v)=(0,0,1)$.
	For $d \ge 2$ it can be proven that the $\kk[s,t]$-module 
	\begin{equation}
		\label{eq_tors_trigr}
		\left[\HH^2_{(u,v)}(R)\right]_{(d,*,-d)} \quad \text{has $\tau_{d-1}$-torsion}
	\end{equation}	
	(see the argument made for ``$\Coker A_{d-1} $'' considered in \cite[Proof of Theorem 1.2]{KATZAMAN}, and note that ``$(\Coker\, A_{d-1} )_{(d,d)}=\Coker\, B_{d-1}$'' is a $\kk[s,t]$-submodule of $\left[\HH^2_{(u,v)}(R)\right]_{(d,*,-d)}$ above).
	
	As in {\sc Part 2}, consider $R$ as a bigraded algebra over $A=\kk[s,t]$, but now set $\bideg(u) = \bideg(v) = (1,0)$ and  $\bideg(x)=\bideg(y)=(1,1)$.	
	The latter bigrading can be obtained by pushing forward the above $\ZZ^3$-grading via the map $\pi : \ZZ^3 \rightarrow \ZZ^2, \, (n_1,n_2,n_3) \mapsto (n_1+n_3, n_2)$.
	In other words, under this new bigrading the graded part $\left[\HH_{(u,v)}^2(R)\right]_{(a,b)}$ can be described with the following direct sum
	\begin{equation}
		\label{eq_dec_bigr}
		\left[\HH_{(u,v)}^2(R)\right]_{(a,b)} = \bigoplus_{\substack{a_1,a_2\in \ZZ\\ a_1+a_2 = a}} \left[\HH^2_{(u,v)}(R)\right]_{(a_1,\,b,\,a_2)}.
	\end{equation}
	Combining \autoref{eq_tors_trigr} and \autoref{eq_dec_bigr} it follows that $\left[\HH_{(u,v)}^2(R)\right]_{(0,*)}$ has $\tau_{d-1}$-torsion for all $d \ge 2$.
	Therefore, we obtain that \autoref{thm_local_cohom_general_fiber_bigrad_mod} fails when setting $j=0$ and $\M = R$ with the current bigrading. 
	The fact that \autoref{thm_local_cohom_general_fiber_bigrad_mod} is not valid in this case is somehow not surprising because 
	$
	\left[\HH_{(u,v)}^2(\kk[s,t,x,y,u,v])\right]_{(0,*)}
	$ 
	is then an \emph{infinitely generated} $A$-module with a set generators 
	$$
	\Big\lbrace
	\frac{1}{u^{\alpha_1}v^{\alpha_2}} s^{\beta_1}t^{\beta_2} x^{\gamma_1}y^{\gamma_2} 
	\mid  \gamma_1 + \gamma_2 - \alpha_1 - \alpha_2 = 0
	\Big\rbrace;
	$$
	this is an opposite situation to \autoref{rem_finite_local_cohom}.
	
	\medskip
	
	{\sc Part 4:}
	On the other hand, similarly to the previous \autoref{section_gen_freeness_local_cohom}, set $S=\kk[s,t,x,y]$ and suppose that $R$ is the standard graded $S$-algebra $R = S[u,v]/\left(sx^2v^2 - (t+s)xyuv+ty^2u^2\right)$.
	Then, \autoref{thm_relative_graded_local_duality}(IV) implies the existence of an element $0 \neq b \in S$ such that 
	$$
	\HH_{(u,v)}^2(R) \otimes_S S_b
	$$ 
	is a free $S_b$-module.
	Additionally, \autoref{cor_dimension_fiber_local_cohom} gives a dense open subset $\mathcal{V} \subset \Spec(S)$ such that the function 
	$$\qqq \in \Spec(S) \mapsto\dim_{k(\qqq)}\left( {\left[\HH_{(u,v)}^2(R)\right]}_{j} \otimes_S k(\qqq) \right)$$ 
	is constant on $\mathcal{V}$ for all $j \in \ZZ$.
\end{example}

\section{Specialization}\label{section_specialization_powers_mods}

In this section, we focus on various specialization environments, where the main results are obtained as an application of the previous sections.

\subsection{Powers of a graded module}

In this part we look at the situation of a given graded module and its symmetric and Rees powers. 
More precisely, we consider the problem of the local behavior of the following gadgets:
\begin{enumerate}[\rm (I)]
	\item Local cohomology of a general fiber for all the symmetric powers of a module.
	\item Local cohomology of a general specialization for all the Rees powers of a module.
\end{enumerate}
The main results in this regard turn out to be obtainable as an application of \autoref{thm_local_cohom_general_fiber_bigrad_mod}.

\medskip

Throughout this section the following simplified setup will be assumed. 
\begin{setup}\rm 
	\label{setup_specialization_powers}
	Let $A$ be a Noetherian reduced ring.
	Let $R$ be a finitely generated graded $A$-algebra which is positively graded (i.e., $\N$-graded).  
	Let $\mm$ be the graded irrelevant ideal $\mm = {\left[R\right]}_+$.
\end{setup}

\subsubsection{Symmetric powers}

Quite generally, if $M$ is a finitely generated $R$-module with a free presentation 
$$
F_1 \xrightarrow{\varphi} F_0 \rightarrow M \rightarrow 0,
$$
associated to a set of generators of $M$ with $s$ elements, then the symmetric algebra of $M$ over $R$ has a presentation
$\Sym_R\left(M\right) \simeq \BBB/\mathcal{L},$
where 
$
\label{eq_BBB_algebra}
\BBB: = R[y_1,\ldots,y_s]
$
is a polynomial ring over $R$ and $\mathcal{L} = I_1 \Big( \left[y_1,\ldots,y_s\right] \cdot \varphi \Big)$.

Now, if $M$ is moreover graded, one has a presentation which is graded,
where, say, $F_0= \bigoplus_{j=1}^sR(-\mu_j)$.
Fix an integer $b \ge \max\{\mu_1,\ldots,\mu_s\}$,
 consider the shifted module $M(b)$ with corresponding graded free presentation
$$
F_1(b) \xrightarrow{\varphi} \bigoplus_{j=1}^sR(b-\mu_j) \rightarrow M(b) \rightarrow 0.
$$
Then, the symmetric algebra $\Sym_R\left(M(b)\right)$  is naturally a bigraded $A$-algebra with the same sort of presentation as above, only now $\BBB$ has a bigraded structure 
 with bidegrees  $\bideg(x)=(\nu,0)$ for any $x \in {\left[R\right]}_\nu \subset \BBB$ and  $\bideg(y_j)=(\mu_j-b,1)$
for $1\le j \le s$.

Clearly, then
\begin{equation}
	\label{eq_shifts_sym_powers}
	{\left[\Sym_R\left(M(b)\right)\right]}_{(j,k)} \simeq {\left[\Sym_R^k(M)\right]}_{j+kb}
\end{equation}
for $k \ge 0, j \in \ZZ$,
where $\Sym_R^k(M)$ denotes the $k$-th symmetric power of $M$.

Let $T=A[x_1,\ldots,x_r]$ be a standard graded polynomial ring mapping onto $R$,
set in addition $\AAA = T[y_1,\ldots,y_s]$, 
with a bigrading given in the same way as for $\BBB$.
Therefore, one has the following surjective bihomogeneous homomorphisms 
\begin{equation}
\label{eq_bigrad_Sym_as_quotient}
\AAA \twoheadrightarrow \BBB \twoheadrightarrow \Sym_R\left(M(b)\right).
\end{equation}

\begin{notation}\rm 
	 If $M$ is a finitely generated graded $R$-module, let $\beta(M)$ denote the maximal degree of an element in a minimal set of generators of $M$.
	Thus, by the graded version of Nakayama's lemma one has 
	$
	\beta(M) := \max\{ k \in \ZZ  \mid  {\left[M/\mm M\right]}_k \neq 0 \}.
	$
\end{notation}

One has the following theorem as an  application of \autoref{thm_local_cohom_general_fiber_bigrad_mod} and the  above considerations.

\begin{theorem}
	\label{thm_general_fibers_symmetric_powers}
	Under {\rm \autoref{setup_specialization_powers}},
	let $M$ be a finitely generated graded $R$-module and  let $j$ be a fixed integer.
	Given $b \in \ZZ$ such that $b\ge \beta(M)$, 
	 there exists a dense open subset $\mathcal{U}_j \subset \Spec(A)$ such that, for all $i\ge 0, k\ge 0$, the function
	\begin{equation*}
	\Spec(A) \longrightarrow \ZZ,\qquad
	\nnn \in \Spec(A) \mapsto\dim_{k(\nnn)}\left( {\left[\HL^i\Big( \Sym_{R \otimes_A k(\nnn)}^k\left(M \otimes_A k(\nnn)\right) \Big)\right]}_{j+kb} \right)
	\end{equation*}
	is locally constant on $\mathcal{U}_j$.
	\begin{proof}
		Drawing on the assumption that $b\ge \beta(M)$ and \autoref{eq_bigrad_Sym_as_quotient}, one applies the statement of \autoref{thm_local_cohom_general_fiber_bigrad_mod} by taking the bigraded module there to be $\Sym_R\left(M(b)\right)$.
		Since one has the isomorphism
		$$
		\Sym_R(M) \otimes_A k(\nnn) \simeq \Sym_{R \otimes_A k(\nnn)}\left(M \otimes_A k(\nnn)\right),
		$$
		the result follows from \autoref{eq_shifts_sym_powers}.
	\end{proof}
\end{theorem}

\subsubsection{Rees powers}

Here the notation and terminology are the ones of \cite{ram1}.
In particular,  the {\em Rees algebra} $\Rees_R(M)$ of a finitely generated $R$-module $M$ having rank is defined to be the symmetric algebra modulo its $R$-torsion.
With this definition,  $\Rees_R(M)$ inherits from $\Sym_R(M)$  a natural bigraded structure.

There are a couple of ways to introduce the $k$-th power of $M$:
	\begin{equation*}
	M^k := {\left[\Rees_R(M)\right]}_{(*,k)} = \bigoplus_{j \in \ZZ} {\left[\Rees_R(M)\right]}_{(j,k)}
	\simeq  \Sym_R^k(M)/\tau_R(\Sym_R^k(M)),
\end{equation*}
where $\tau_R$ denotes $R$-torsion.

In addition, there is an $R$-embedding 
$
\iota_k : M^k = {\big[\Rees_R(M)\big]}_{(*,k)}  \hookrightarrow {\big[R[t_1,\ldots,t_m]\big]}_{(*,k)}
$
out of an embedding 
\begin{equation}
\label{eq_embedding_Rees}
\Rees_R(M) \hookrightarrow \Sym_R(F) \simeq R[t_1,\ldots,t_m],
\end{equation}
induced by a given embedding of $M^1=M/\tau_R(M)$ into a free $R$-module $F$ of rank equal to the rank of $M$.

\begin{definition}\rm
	\label{def_spec_powers_mods}
	Let $M$ be a finitely generated graded $R$-module having rank. 
	For $\nnn \in \Spec(A)$ and $k \ge 0$, the {\em specialization of $M^k$ with respect to $\nnn$} is the following $R \otimes_A k(\nnn)$-module
	$$
	\SSS_\nnn(M^k): = \IM\Big( \iota_k \otimes_A k(\nnn): M^k \otimes_A k(\nnn) \rightarrow {\big[(R \otimes_A k(\nnn))[t_1,\ldots,t_m]\big]}_{(*,k)} \Big).
	$$
	If no confusion arises, one sets $\SSS_\nnn(M):=\SSS_\nnn(M^1)$.
\end{definition}

\begin{proposition}
	\label{lem_specialization_powers_fiber}
	Let $M$ be a finitely generated graded $R$-module having rank. 
	Then, there is a dense open subset $V \subset \Spec(A)$ such that, for all $\nnn \in V$ and $k \ge 0$, one has
	$$
	\SSS_\nnn(M^k) \simeq M^k \otimes_A k(\nnn).
	$$	
	In particular, $\SSS_\nnn(M^k)$ is independent of the chosen embedding $M^1\hookrightarrow F$. 
	\begin{proof}
		From \autoref{eq_embedding_Rees}, consider the short exact sequence
		$$
		 0 \rightarrow \Rees_R(M) \rightarrow R[t_1,\ldots,t_m] \rightarrow \frac{R[t_1,\ldots,t_m]}{\Rees_R(M)} \rightarrow 0.
		$$
		By using \autoref{thm_generic_freeness_mult_grad}(i)  (as applied in the notation there with $M=\BBB=R[t_1,\ldots,t_m]$, $E=\Rees_R(M)$ and $H=0$) choose $a \in A$ avoiding the minimal primes of $A$ such that $\frac{R[t_1,\ldots,t_m]}{\Rees_R(M)} \otimes_{A} A_a$ is a projective $A_a$-module.
		So, the result follows by setting $V = D(a) \subset \Spec(A)$.
	\end{proof}
\end{proposition}

\begin{corollary}
	\label{cor_specialization_module}
	Under {\rm \autoref{setup_specialization_powers}},
	let $M$ be a finitely generated graded $R$-module having rank.
	Then, there exists a dense open subset $\mathcal{U} \subset \Spec(A)$ such that, for all $i \ge 0, j \in \ZZ$, the function
	\begin{equation*}
	\Spec(A) \longrightarrow \ZZ, \qquad
	\nnn \in \Spec(A) \mapsto\dim_{k(\nnn)}\Bigg( {\left[\HL^i\big(\SSS_\nnn(M)\big)\right]}_{j} \Bigg)
	\end{equation*}
	is locally constant on $\mathcal{U}$.
	\begin{proof}
		It follows from \autoref{cor_dimension_fiber_local_cohom} and \autoref{lem_specialization_powers_fiber}.
	\end{proof}
\end{corollary}

Next is the principal result about the specialization of the Rees powers of a graded module. The proof is again short because it is downplayed by the use of previous theorems.

\begin{theorem}
	\label{thm_general_fibers_powers}
Under {\rm \autoref{setup_specialization_powers}},
let $M$ be a finitely generated graded $R$-module having rank. Fix an integer $j \in \ZZ$ and let $b \in \ZZ$ be an integer such that $b\ge \beta(M)$. 
	Then, there exists a dense open subset $\mathcal{U}_j \subset \Spec(A)$ such that, for all $i\ge 0, k\ge 0$, the function
	\begin{equation*}
	\Spec(A) \longrightarrow \ZZ, \qquad
	\nnn \in \Spec(A) \mapsto\dim_{k(\nnn)}\left( {\left[\HL^i\Big( \SSS_\nnn(M^k) \Big)\right]}_{j+kb} \right)
	\end{equation*}
	is locally constant on $\mathcal{U}_j$.
	\begin{proof}
		One extends \autoref{eq_bigrad_Sym_as_quotient} to the following surjective bihomogeneous homomorphisms 
		$$
		\AAA \twoheadrightarrow \BBB \twoheadrightarrow \Sym_R(M(b)) \twoheadrightarrow \Rees_R(M(b)).
		$$
		Note that 
		$
		{\left[\Rees_R\left(M(b)\right)\right]}_{(j,k)} \simeq {\left[M^k\right]}_{j+kb}
		$
		for all $j \in \ZZ, k \ge 0$.
		One sets $\Rees_R(M(b))$ to be the bigraded module in the statement of \autoref{thm_local_cohom_general_fiber_bigrad_mod}.
		Then,  let $U_j \subset \Spec(A)$ be a dense open subset  obtained from \autoref{thm_local_cohom_general_fiber_bigrad_mod}.
		Let $V \subset  \Spec(A)$ be a dense open subset from \autoref{lem_specialization_powers_fiber}.
		Therefore, the result follows by setting $\mathcal{U}_j = U_j \cap V$.
	\end{proof}
\end{theorem}

\subsection{Rational maps and the saturated special fiber}
\label{section_rational_maps}

In this section one revisits the problem of specialization of rational maps, as studied in \cite{SPECIALIZATION_RAT_MAPS}.
We recover some of the results there as a consequence of \autoref{thm_general_fibers_powers} and \cite[Corollary 2.12]{MULTPROJ}.
Quite naturally, one also studies the saturated special fiber ring and the $j$-multiplicity of a general specialization of an ideal.

For the basics of rational maps with source and target projective varieties defined over an arbitrary Noetherian  domain, the reader is referred to \cite[Section 3]{SPECIALIZATION_RAT_MAPS}.

Throughout this section the following setup is used. 

\begin{setup}\rm
	\label{setup_rat_maps}
	Let $A$ be a Noetherian domain and $R$ be the standard graded polynomial ring $R=A[x_0,\ldots,x_r]$. 
	Fix homogeneous elements $\{g_0,\ldots,g_s\} \subset R$   of the same degree $d>0$ and let  $\GG :  \PP_A^r \dashrightarrow \PP^s_A$ denote the corresponding rational map given by the representative $\mathbf{g}=(g_0:\cdots:g_s)$. 
	Set $\mm={\left[R\right]}_+ = (x_0,\ldots,x_r) \subset R$.

	We specialize this rational map as follows.
	Given $\nnn \in \Spec(A)$, take the rational map $\GG(\nnn): \PP_{k(\nnn)}^r \dashrightarrow \PP_{k(\nnn)}^s$ with representative
	$$\pi_\nnn\left({\mathbf{g}}\right)=(\pi_\nnn({g_0}):\cdots:\pi_\nnn({g_s})),
	$$
	where $\pi_\nnn({g_i})$ is the image of $g_i$ under the canonical map $\pi_\nnn : R \rightarrow R \otimes_A k(\nnn)$.
	
	Set $I=(g_0,\ldots,g_s) \subset R$ and note that $\SSS_\nnn(I)=(\pi_\nnn({g_0}),\ldots,\pi_\nnn({g_s})) \subset R \otimes_A k(\nnn)$ and that $\SSS_\nnn(I^k) = \SSS_\nnn(I)^k \subset R \otimes_A k(\nnn)$, for $\nnn \in \Spec(A), k \ge 0$ (see \autoref{def_spec_powers_mods}).
	
	Given $\nnn \in \Spec(A)$, let $Y \subset \PP_A^s$ and  $Y(\nnn) \subset \PP_{k(\nnn)}^s$ denote the respective closed images of $\GG$ and of $\GG(\nnn)$.  
\end{setup}

	Recall that the rational map  $\GG(\nnn)$ is \textit{generically finite} if one of the following equivalent conditions is satisfied:
	\begin{enumerate}[(i)]
		\item The field extension $K(Y(\nnn)) \hookrightarrow K(\PP_{k(\nnn)}^r)$ is finite, where $K(\PP_{k(\nnn)}^r)$ and $K(Y(\nnn))$ denote the fields of rational functions of $\PP_{k(\nnn)}^r$ and $Y(\nnn)$, respectively.
		\item $\dim(Y(\nnn)) = \dim(\PP_{k(\nnn)}^r)=r$.
		\item The \textit{analytic spread}  
		$
		\ell\big(\SSS_\nnn(I)\big):=\dim\Big(\Rees_{R \otimes_A k(\nnn)}\left(\SSS_\nnn(I)\right)/\mm\Rees_{R \otimes_A k(\nnn)}\left(\SSS_\nnn(I)\right) \Big)$
		of $\SSS_\nnn(I)$ attains the maximum possible value
		 $\dim\left(R \otimes_A k(\nnn)\right)=r+1.
		$
	\end{enumerate} 
	The \textit{degree} of $\GG(\nnn)$ is defined as
	$
	\deg(\GG(\nnn)):=\left[K(\PP_{k(\nnn)}^r):K(Y(\nnn))\right].
	$

\begin{definition}[\cite{ACHILLES_MANARESI_J_MULT}] \rm 
	For any $\nnn \in \Spec(A)$ and any homogeneous ideal $J \subset R \otimes_A k(\nnn)$, the \emph{$j$-multiplicity} of $J$ is given by 
	$$
	j(J) \,:=\, r
	!\lim\limits_{n\rightarrow\infty} \frac{\dim_{k(\nnn)}\Big(\HH_{\mm}^0\left(J^n/J^{n+1}\right)\Big)}{n^{r}}.
	$$
\end{definition}

\begin{definition}[\cite{MULTPROJ}] \rm 
	For any $\nnn \in \Spec(A)$ and any homogeneous ideal $J \subset R \otimes_A k(\nnn)$ generated by elements of the same degree $d>0$, the \emph{saturated special fiber ring} of $J$ is given by 
	$$
	\widetilde{\mathfrak{F}_{_{R \otimes_A k(\nnn)}}}\left(J\right) \,:=\, \bigoplus_{n=0}^\infty {\left[\big(J^n:\mm^\infty\big)\right]}_{nd}.
	$$
\end{definition}

Next is the main result of this section.

\begin{theorem}
	Under {\rm \autoref{setup_rat_maps}}, assume in addition that   $\GG\left((0)\right)$ is generically finite.
	Then, there exists a dense open subset $\mathcal{U} \subset \Spec(A)$ such that $\GG(\nnn)$ is generically finite for any $\nnn \in \mathcal{U}$ and the functions
	\begin{enumerate}[\rm (i)]
		\item $\nnn \in \Spec(A) \mapsto \deg\left(\GG(\nnn)\right)$,
		\item $\nnn \in \Spec(A) \mapsto \deg_{\PP_{k(\nnn)}^s}\left(Y(\nnn)\right)$,
		\item $\nnn \in \Spec(A) \mapsto e\left(\widetilde{\mathfrak{F}_{_{R \otimes_A k(\nnn)}}}\left(\SSS_\nnn(I)\right)\right)$ \mbox{and}
		\item 	$\nnn \in \Spec(A) \mapsto j\left(\SSS_\nnn(I)\right)$
	\end{enumerate}
	are constant on $\mathcal{U}$.
	\begin{proof}
		We first argue for (i) and (ii).
		By \autoref{lem_specialization_powers_fiber} there exists a dense open subset $U \subset \Spec(A)$ such that 
		$$
		\Rees_{_{R \otimes_A k(\nnn)}}\big(\SSS_\nnn(I)\big)=\bigoplus_{k=0}^\infty\SSS_\nnn(I)^k \,\simeq\, \Rees_R(I) \otimes_A k(\nnn) \simeq \bigoplus_{k=0}^\infty I^k \otimes_A k(\nnn)
		$$
		for all $\nnn \in U$.
		One has an isomorphism $Y(\nnn) \simeq \Proj\Big(k(\nnn)\big[\pi_\nnn({g_0}),\ldots,\pi_\nnn({g_s})\big]\Big)$ (see, e.g., \cite[Definition-Proposition 3.12]{SPECIALIZATION_RAT_MAPS}).
		By restricting to the zero graded part in the $R$-grading, we obtain the following isomorphisms of graded $k(\nnn)$-algebras
		$$
		k(\nnn)\big[\pi_\nnn({g_0}),\ldots,\pi_\nnn({g_s})\big] \,\simeq\, {\left[\Rees_{_{R \otimes_A k(\nnn)}}\big(\SSS_\nnn(I)\big)\right]}_0 \,\simeq\, {\left[\Rees_R(I)\right]}_0 \otimes_A k(\nnn)
		$$
		for any $\nnn \in U$ (as before in \autoref{eq_R_grad_bigrad_convention}, one uses the notation ${\left[\Rees_R(I)\right]}_0=\bigoplus_{\nu =0}^\infty{\left[\Rees_R(I)\right]}_{(0,\nu)}$).
		
		By \autoref{thm_generic_freeness_mult_grad}(ii), as applied with $\mathfrak{M}:=\Rees_R(I)$, there is an element $0\neq a \in A$ such that all the graded components of  ${\left[\Rees_R(I)\right]}_0 \otimes_A A_a$ are free $A_a$-modules.
		Set $V=D(a) \subset \Spec(A)$.
		Since $\GG{\left((0)\right)}$ is generically finite,  one has  $\dim\big({\left[\Rees_R(I)\right]}_0 \otimes_A k((0))\big)=\dim\left(R \otimes_{A} k((0))\right)$, and so it follows that 
		$$
		\dim\big({\left[\Rees_R(I)\right]}_0 \otimes_A k(\nnn)\big)=\dim\big({\left[\Rees_R(I)\right]}_0 \otimes_A k((0))\big)=\dim\left(R((0))\right)=\dim\left(R \otimes_A k(\nnn)\right)
		$$
		and that
		$$
		\deg_{\PP_{k(\nnn)}^s}\left(Y(\nnn)\right) = e\big({\left[\Rees_R(I)\right]}_0 \otimes_A k(\nnn)\big)=e\big({\left[\Rees_R(I)\right]}_0 \otimes_A k((0))\big) = \deg_{\PP_{k((0))}^s}\left(Y_{(0)}\right)
		$$
		for any $\nnn \in U \cap V$.
		
		For any $\nnn \in U \cap V$,  \cite[Corollary 2.12]{MULTPROJ} yields the formula 
		$$
		\deg_{\PP_{k(\nnn)}^s}\left(Y(\nnn)\right)\left(\deg\left(\GG(\nnn)\right)-1\right) = r!\lim_{k\rightarrow \infty} \frac{\dim_{k(\nnn)}\left({\left[\HL^1\Big( \SSS_\nnn(I^k) \Big)\right]}_{kd}\right)}{k^{r}}.
		$$
		Let $W \subset \Spec(A)$ be a dense open subset obtained from \autoref{thm_general_fibers_powers} with  $M:=I(d)$.
		It then follows that the function
		$$
		\nnn \in \Spec(A) \mapsto \deg_{\PP_{k(\nnn)}^s}\left(Y(\nnn)\right)\left(\deg\left(\GG(\nnn)\right)-1\right)
		$$ 
		is constant on $U \cap V \cap W.$
		
		So, the result follows by taking $\mathcal{U} = U \cap V \cap W$.
		
		\smallskip
		
		(iii) It follows from parts (i), (ii) and \cite[Theorem 2.4]{MULTPROJ}.
		
		\smallskip
		
		(iv) It follows from parts (i), (ii) and \cite[Theorem 5.3]{KPU_blowup_fibers}.
	\end{proof}
\end{theorem}

\subsection{Numerical invariants}
\label{section_num_invariants}

The goal is to show that dimension, depth, $a$-invariants and regularity of a module are locally constant under tensor product with a general fiber and general specialization. 
As a side-result, we provide a slight improvement of the upper semi-continuity theorem (see \cite[Chapter III, Theorem 12.8]{HARTSHORNE})) for the dimension of sheaf cohomology of a general fiber.

For a finitely generated graded $R$-module $M$ the $i$-th $a$-invariant is defined as
\begin{equation}
\label{eq_a_invariants}
a^i(M) := \begin{cases}
\max\big\{n \mid {\left[\HL^i(M)\right]}_n \neq 0\big\}  \quad \text{if } M \neq 0\\
-\infty \qquad\qquad \qquad\qquad\qquad\; \text{if } M = 0
\end{cases}
\end{equation}
and the Castelnuovo--Mumford regularity is given by 
\begin{equation}
\label{eq_regularity}
\reg(M) := \max\big\{ a^i(M) + i \mid i \ge 0 \big\}.
\end{equation}

We first state the local behavior of the numerical invariants for the fibers.

\begin{proposition}
	\label{thm_general_fibers_symmetric_powers_homolog_prop}
	Under {\rm \autoref{setup_specialization_powers}},
	let $M$ be a finitely generated graded $R$-module. 
	Then, there exists a dense open subset $\mathcal{U} \subset \Spec(A)$ such that the functions
	\begin{enumerate}[\rm (i)]
		\item $\nnn \in \Spec(A) \mapsto \dim\left( M \otimes_A k(\nnn) \right)$,
		\item 
		$\nnn \in \Spec(A) \mapsto \depth\left( M \otimes_A k(\nnn)\right) $,
		\item 
		$\nnn \in \Spec(A) \mapsto a^i\left(M \otimes_A k(\nnn) \right)$ for $i \ge 0$, \, \mbox{and}
		\item
		$\nnn \in \Spec(A) \mapsto \reg\left( M \otimes_A k(\nnn)\right)$
	\end{enumerate}
	are locally constant on $\mathcal{U}$.
	\begin{proof}
		It follows from \autoref{cor_dimension_fiber_local_cohom}, \cite[Corollary 6.2.8]{Brodmann_Sharp_local_cohom}, \autoref{eq_a_invariants} and \autoref{eq_regularity}. 
	\end{proof}
\end{proposition}

Next is the local behavior of the numerical invariants for the specialization.

\begin{proposition}
	\label{thm_general_fibers_powers_homolog_prop}
	Under {\rm \autoref{setup_specialization_powers}},
	let $M$ be a finitely generated graded $R$-module having rank. 
	Then, there exists a dense open subset $\mathcal{U} \subset \Spec(A)$ such that the functions
	\begin{enumerate}[\rm (i)]
		\item $\nnn \in \Spec(A) \mapsto \dim\left(\SSS_\nnn\left(M\right) \right)$,
		\item 
		$\nnn \in \Spec(A) \mapsto \depth\left(\SSS_\nnn\left(M\right) \right)$,
		\item 
		$\nnn \in \Spec(A) \mapsto a^i\left(\SSS_\nnn\left(M\right) \right)$ for $i \ge 0$ \mbox{and}
		\item
		$\nnn \in \Spec(A) \mapsto \reg\left(\SSS_\nnn\left(M\right) \right)$
	\end{enumerate}
	are locally constant on $\mathcal{U}$.
	\begin{proof}
		It follows from \autoref{cor_specialization_module},  \cite[Corollary 6.2.8]{Brodmann_Sharp_local_cohom}, \autoref{eq_a_invariants} and \autoref{eq_regularity}. 
	\end{proof}
\end{proposition}

An additional outcome is a slight improvement of the upper semicontinuity theorem.

\begin{proposition}
	Let $A$ be denote a  reduced Noetherian ring.
	Let $R$ be a standard graded finitely generated $A$-algebra and  $X: = \Proj(R)$.
	Given  a finitely generated graded $R$-module $M$, 
	 there exists a dense open subset $\mathcal{U} \subset \Spec(A)$ such that, for all $i\ge 0, n \in \ZZ$, the function
	\begin{equation*}
	\Spec(A) \longrightarrow \ZZ, \qquad
	\nnn \in \Spec(A) \mapsto\dim_{k(\nnn)}\Big( \HH^i\big(X \times_A k(\nnn), \widetilde{M(n)} \otimes_{A} k(\nnn)\big) \Big)
	\end{equation*}
	is locally constant on $\mathcal{U}$. 
	\begin{proof}
		For $i \ge 1$, one has that $\HH^i\big(X \times_A k(\nnn), \widetilde{M(n)} \otimes_{A} k(\nnn)\big) \simeq {\left[\HL^{i+1}(M\otimes_A k(\nnn))\right]}_n$ (see, e.g., \cite[Theorem A4.1]{EISEN_COMM}), and so in this case the result is obtained directly from \autoref{cor_dimension_fiber_local_cohom}.
		
		For $i = 0$, one has the short exact sequence 
		\begin{align*}
			0 \rightarrow {\left[\HL^{0}(M\otimes_A k(\nnn))\right]}_n \rightarrow {\left[M\otimes_A k(\nnn)\right]}_n \rightarrow \HH^0&\big(X \times_A k(\nnn), \widetilde{M(n)} \otimes_{A} k(\nnn)\big) \\
			&\rightarrow {\left[\HL^{1}(M\otimes_A k(\nnn))\right]}_n \rightarrow 0
		\end{align*}
		(see, e.g., \cite[Theorem A4.1]{EISEN_COMM}).
		From \autoref{cor_generic_projective}, there is a dense open subset $U \subset \Spec(A)$ such that $\dim_{k(\nnn)}\left({\left[M\otimes_A k(\nnn)\right]}_n\right)$ is locally constant for all $\nnn \in U$.
		Take a dense open subset $V \subset \Spec(A)$ given as in \autoref{cor_dimension_fiber_local_cohom}.
		So, the result follows in both cases by setting $\mathcal{U} = U \cap V$.
	\end{proof}
\end{proposition}

\section*{Acknowledgments}

The authors are indebted to the referee for a constructive criticism of the paper.

\bibliographystyle{elsarticle-num} 

\end{document}